\documentclass[oneside]{amsart}   	% use "amsart" instead of "article" for AMSLaTeX format
\usepackage{geometry}               % See geometry.pdf to learn the layout options. There are lots.
\usepackage{graphicx,color}						
\usepackage[dvipsnames]{xcolor}
\usepackage{amssymb,amsmath,amsthm,amsfonts, mathrsfs}
 
\usepackage{bm}                     % mathbf greek
\usepackage{enumitem}
\usepackage[english]{babel}
\usepackage[latin1]{inputenc}

\usepackage{pgf}
\usepackage[sc]{mathpazo}

\newtheorem{proposition}{Proposition}[section]
\newtheorem{theorem}[proposition]{Theorem}
\newtheorem{corollary}[proposition]{Corollary}
\newtheorem{lemma}[proposition]{Lemma}

\theoremstyle{definition}
\newtheorem{definition}[proposition]{Definition}
\theoremstyle{remark}
\newtheorem{remark}[proposition]{Remark}

\newcommand{\eps}{\varepsilon}

\newcommand{\N}{{\mathbb{N}}}
\newcommand{\Z}{{\mathbb{Z}}}

\newcommand{\R}{{\mathbb{R}}}
\newcommand{\C}{{\mathbb{C}}}

\newcommand{\from}{\colon}
\newcommand{\loc}{{\mathrm{loc}}}

\newcommand{\bx}{{\mathbf{x}}}

\newcommand{\Ncal}{{\mathcal{N}}}

\newcommand{\Rcal}{{\mathcal{R}}}
\newcommand{\Scal}{{\mathcal{S}}}
\newcommand{\Tcal}{{\mathcal{T}}}
\newcommand{\Ucal}{{\mathcal{U}}}

\newcommand{\rot}{{\mathrm{rot}}}

\DeclareMathOperator{\supp}{supp}

\DeclareMathOperator{\sign}{sign}

\DeclareMathOperator{\real}{Re}
\DeclareMathOperator{\imag}{Im}

\newcommand{\cc}{S}

\newcommand{\modB}{{\Theta}}

\newcommand{\bound}{{\varphi}}
\newcommand{\var}{{t}}

\title{Rotating Spirals in segregated reaction-diffusion systems}
\author{Ariel Salort, Susanna Terracini, Gianmaria Verzini and Alessandro Zilio}
\date{\today}

\email{asalort@dm.uba.ar}
\email{susanna.terracini@unito.it}
\email{gianmaria.verzini@polimi.it}
\email{alessandro.zilio@u-paris.fr}

\address{Ariel Salort \newline \indent
Departamento. de Matem\'atica,
           FCEyN, UBA \newline \indent
           Ciudad Universitaria, Pab. 1 (1428)
           Of. 2106 \newline \indent
           Buenos Aires, Argentina}

\address{Susanna Terracini \newline \indent
 Dipartimento di Matematica ``Giuseppe Peano'', \newline \indent
Universit\`a di Torino, \newline \indent
Via Carlo Alberto, 10,
10123 Torino, Italy}

\address{Gianmaria Verzini\newline \indent
 Dipartimento di Matematica ``Francesco Brioschi'', \newline \indent
Politecnico di Milano, \newline \indent
Piazza Leonardo da Vinci, 32
20133 Milano, Italy}

\address{Alessandro Zilio \newline \indent
Université de Paris and Sorbonne Université,  \newline \indent
CNRS, Laboratoire Jacques-Louis Lions (LJLL),  \newline \indent
F-75006 Paris, France}

\begin{document}
\begin{abstract}
We give a complete characterization of the boundary traces $\varphi_i$ ($i=1,\dots,K$) 
supporting spiraling waves, rotating with a given angular speed $\omega$, which appear as 
singular limits of competition-diffusion systems of the type
\begin{equation*}
	\begin{cases}
		\partial_t u_i -\Delta u_i  = \mu u_i -\beta u_i \sum_{j \neq i} a_{ij} u_j &\text{in } \Omega
		\times\R^+ \\
		u_i = \varphi_i &\text{on 
		$\partial\Omega\times\R^+$}\\
		u_i(\bx,0) = u_{i,0}(\bx) &\text{for $ \bx \in \Omega$}
	\end{cases}
\end{equation*}
as $\beta\to +\infty$. Here $\Omega$ is a rotationally invariant planar set and $a_{ij}>0$ 
for every $i$ and $j$. We tackle also the homogeneous Dirichlet and Neumann boundary 
conditions, as well as entire solutions in the plane. As a byproduct of our analysis we 
detect explicit families of eternal, entire solutions of the pure heat equation, parameterized by 
$\omega\in\R$, which reduce to homogeneous harmonic polynomials for $\omega=0$.  
\end{abstract}
\maketitle

\section{Introduction}

This paper deals with existence, uniqueness and qualitative properties of rotating spiraling waves arising in the singular limit of reaction-diffusion systems,  when the interspecific competition rates become infinite.  More precisely, we are concerned with the singular limits, 
as $\beta\to +\infty$, of the following model problem involving $K\ge3$ species competing in the plane:

\begin{equation}\label{eqn sys beta parab}
	\begin{cases}
		\partial_t u_i -\Delta u_i  = f_i(u_i) -\beta u_i \sum_{j \neq i} a_{ij} u_j &\text{in }
		\Omega\times\R^+\\
		u_i = \varphi_i &\text{on }
		\partial\Omega\times\R^+\\
		u_i(\bx,0) = u_{i,0}(\bx) &\text{for $ \bx \in \Omega$}.
	\end{cases}
\end{equation}
Here $\Omega \subset \R^2$ has a smooth boundary,  $u_i=u_i(\bx,t)$ represents the density of the $i$-th species ($1\leq i \leq K$), whose internal dynamic is described by the function $f_i$. The positive numbers $\beta a_{ij}$  account for the interspecific  competition rates, so that the interaction has a repulsive character. The boundary data $\varphi_i$ are positive and 
segregated, i.e.\ 
$\varphi_i\varphi_j\equiv0$ for $j\neq i$.

As already mentioned, we are concerned with  the limit case of strong competition, that is when the parameter $\beta$ goes to $+\infty$, while the positive coefficients $a_{ij}$ remain fixed.  In this case it is known that the densities $u_i$ \emph{segregate}, in the sense that they converge uniformly to limit densities satisfying $u_iu_j\equiv0$ for $j\neq i$; 
hence a pattern arises, and the common nodal set (where all densities vanish simultaneously) can be considered as a free boundary (see  \cite{MR2529504,CoTeVe_AM2005,CoTeVe_IUMJ2005,MR2384550} for steady states and \cite{MR2846360,MR2863857,MR1900331,MR2595199}  for time varying solutions). For such segregated limit profiles, the interface conditions are expressed by two systems of differential inequalities  which play a fundamental role in our work:  
\begin{equation}\label{esse_par}
  \partial_t u_i -\Delta u_i\leq f_i(u_i),\quad \partial_t \widehat u_i-\Delta \widehat u_i\geq \widehat f_i(\widehat u_i),
\end{equation}\label{eq:esset}
where the differential inequalities are understood in variational sense and
\begin{equation} \label{cappuccio}
\widehat u_i=u_i-\sum_{j\neq i}\dfrac{a_{ij}}{a_{ji}}u_j,
\qquad
\widehat f_i(\widehat u_i) =f( u_i)-\sum_{j\neq i}\dfrac{a_{ij}}{a_{ji}}f(u_j).
\end{equation}
These inequalities incorporate  the transmission conditions at the free boundary, that is the closure of  the interfaces $\partial\{u_i > 0\}\cap \partial\{u_j > 0\}$, which separate the supports of $u_i$ and $u_j$ at any fixed time $t$.

For planar stationary solutions, the structure of the free boundary has been the object of several papers. 
In the case of symmetric interactions ($a_{ij}=
a_{ji}$ for every $i$ and $j$) it is composed by a regular part, a collection of  smooth curves, meeting at a locally finite number of (singular) clustering points, with definite tangents (see  \cite{MR2529504,CoTeVe_AM2005, CoTeVe_IFB2006, MR2483815}). On the other hand, the asymetric 
case has been treated only more recently in \cite{MR4020313}: while the topological structure 
of the free boundary is analogous to the symmetric 
case (smooth curves meeting at isolated singular 
points), the geometric description differs strongly 
in a neighborhood of each singular point, where 
the nodal lines meet with logarithmic spiraling asymptotics. 

Going back to time-dependent systems, rotating spiraling patterns have been detected numerically in the case of three competing populations in \cite{MR2776460}.  Driven by this phenomenology, in this paper we seek rotating spirals, that is  rigidly rotating waves which are steady states of \eqref{esse_par} in a reference frame spinning with frequency $\omega$; such solutions satisfy $\partial_t u_i =\omega \partial _\theta u_i$ in a disk, subject to boundary conditions which are prescribed in the rotating frame, and exhibiting spiraling 
interfaces near the origin. Hence, in comparison with the literature, our work tackles the segregation problem from a new perspective, that is the existence of limit segregated profiles satisfying additional qualitative properties or shadowing some given shapes. On the other hand, 
the literature on other aspects of segregation triggered by strong competition, starting from pioneering works by Dancer and Du \cite{MR1312772, MR1312773},  is now very vast and it is impossible to give a complete account of it here; besides the papers quoted above, we quote a few more recent ones such as \cite{VeZi_CPDE2014,MR3651893, MR3864295,MR4026185,MR3948936,MR4298757} and we refer the interested reader to the references therein.

The rotating spiral shapes we investigate evoke some other typical example of spatio-temporal patterns arising in reaction-diffusion systems in planar domains: the spiral waves. In the simplest case, these are stationary waves in a rotating frame, while modulated spiraling waves may emanate from rigidly rotating ones in some circumstances.  Such waves arise in different models and appear in the literature about reaction-diffusion systems in contexts different from singular perturbation problems (see e.g.\ \cite{MR1438374,MR2318665,sandstede2021spiral} and references therein).  As far as we know, this is the first study on spiraling rotating waves for segregated limit profiles of competition-diffusion systems.   We also mention that  spiraling  interfaces arise in free boundary problems in entirely different contexts \cite{MR4047645}.

To construct eternal solutions of spiraling type to the limit system \eqref{esse_par}, in this paper we deal with suitable classes of reactions $f_i$ and  boundary conditions. More precisely, let us consider identical, linear reactions in the unit ball (centered at 
$\mathbf{0}$):
\[
\Omega = B, \qquad f_i(u)=\mu u,\text{ for 
some $\mu \in \R$.}
\]
We insert into \eqref{esse_par} the rotating wave ansatz
\[
	u_i(\bx,t) = u_i(\mathcal{R}_{\omega t} \bx),
\]
where
\[
	\mathcal{R}_{\omega t} = \begin{pmatrix}
		\cos( \omega t) & -\sin( \omega t)\\
		\sin( \omega t) & \cos( \omega t)
	\end{pmatrix}
\]
is the rotation matrix of angular speed $\omega$, and we
obtain the stationary system of inequalities
\begin{equation}\label{eqn Somega}
	\begin{cases}
		-\Delta u_i + \omega \bx^\perp \cdot \nabla u_i \leq \mu u_i &\text{in $B$}\\
		-\Delta \widehat u_i  + \omega \bx^\perp \cdot \nabla \widehat u_i \geq 
		\mu\widehat u_i &\text{in $B$}\\
   \phantom{-\Delta}u_i\cdot u_j=0&\text{for }i\neq j,
	\end{cases}
\end{equation}
where $\bx^\perp = \Rcal_{\pi/2}\bx$ and $\widehat u_i$ is defined in \eqref{cappuccio}. 
It is worth noting that, despite appearances, this system is strongly nonlinear and has to be tackled as a free boundary problem.

We are interested in solutions of \eqref{eqn Somega} whose nodal set consists 
in smooth arcs, emanating from $\partial B$ and spiraling towards $\mathbf{0}$, which is the 
unique singular point of the free boundary. In this way, each arc is a smooth interface between 
two adjacent densities, and the origin is 
the only point with higher multiplicity (see Fig.~\ref{fig:intro}). In 
\begin{figure}[t]
\begin{center}
\includegraphics[width=.55\textwidth]
{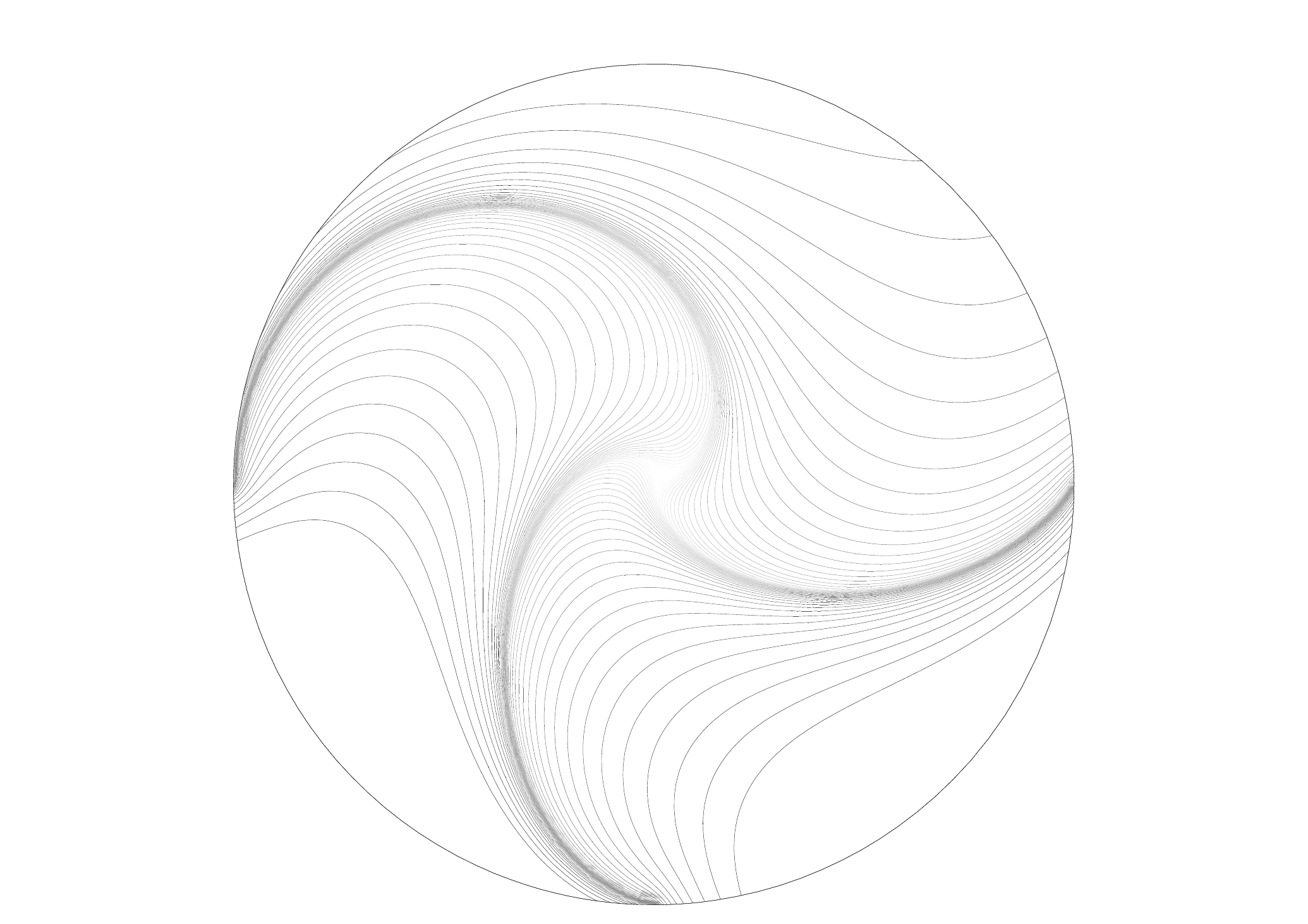} \hspace{-2cm}
\includegraphics[width=.55\textwidth]
{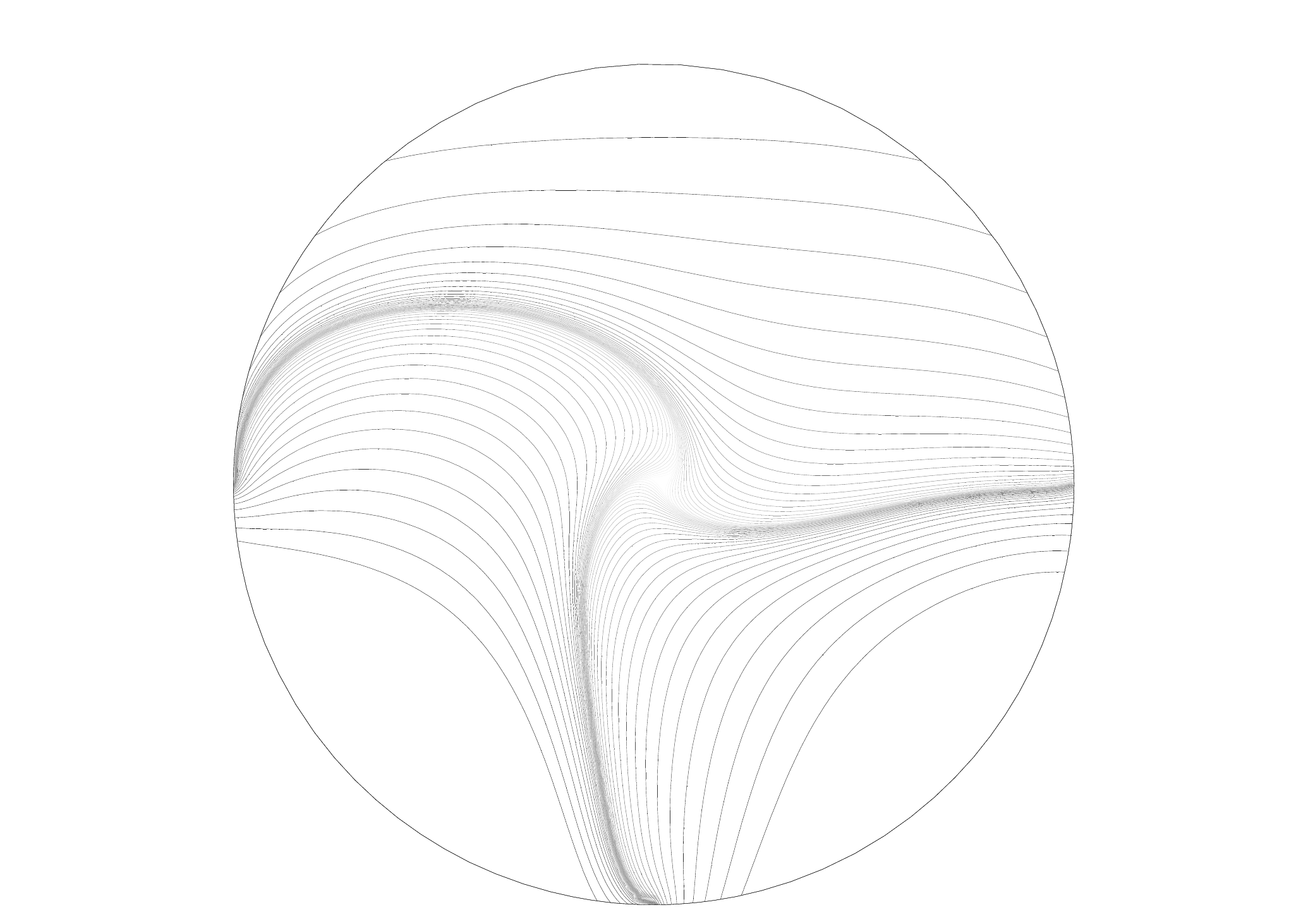}
\caption{Counter lines of a numerical simulation (obtained in FreeFem++ \cite{MR3043640}) in the case of $K=3$ densities, with asymmetric competition such that $\frac{a_{12}}{a_{21}}=\frac{a_{23}}{a_{32}}= 
\frac{a_{31}}{a_{13}}=10$, and reaction term $\mu=0$. The angular velocity is $\omega=3$ for the picture on the left (counterclockwise spin) and $\omega = -3$ for the picture on the right (clockwise spin). In both cases we obtain a unique singular point at the center of the circle by choosing the same boundary conditions, that verify the necessary and sufficient condition in Theorem \ref{thm:main_intro} (see equation \eqref{eq:compX3}). The rotation affects the shape of the spirals, but not their asymptotic behavior close to the center. This is part of the content of Theorem \ref{thm:main_intro}.}
\label{fig:intro}
\end{center}
\end{figure}
this framework we provide a complete description of the non-homogeneous Dirichlet problem associated with \eqref{eqn Somega}.

Let us consider a $K$-tuple  $(\varphi_1,\dots,\varphi_K)$ of segregated boundary traces. 
Precisely, we assume that, for every $i=1,\dots,K$,
\begin{equation}\label{eq:ass_fi_i}
\begin{cases}
\varphi_i\in C^{0,1}(\partial B),\ \varphi_i \geq 0,\smallskip\\
\{\bx:\varphi_i(\bx)>0\}\text{ are connected, non-empty and disjoint arcs,}\smallskip\\
\bigcup_{i}\supp\varphi_i = \partial B.
\end{cases}
\end{equation}
Up to relabelling, we can assume that the traces $\varphi_i$ are labeled in counterclockwise 
order. 

In general, it is not reasonable to expect that any choice of the boundary data 
provides a solution of \eqref{eqn Somega} with a unique singular point at $\mathbf{0}$. 
Indeed, we show that this happens exactly for an explicit subset having codimension 
$K-1$ in the space of traces. Let $s=(s_1,\dots,s_K)\in\R^K$, with $s_i>0$ for all $i$, 
and let us consider the class of functions
\begin{equation}\label{eq:S_om}
\Scal_{\rot} =
 \left\{U=(u_1,\cdots,u_K)\in (H^1(B))^K:\,
  \begin{array}{l}
   u_i\geq0\text{ satisfy \eqref{eqn Somega},}\smallskip\\  
   u_i=s_i\varphi_i\mbox{ on }\partial B
  \end{array}
 \right\}.
\end{equation}
To state our main result we introduce the parameter
\begin{equation}\label{eq:def_alfa}
\alpha = \frac{1}{2\pi}\ln\left(\frac{a_{12}}{a_{21}}\cdot\frac{a_{23}}{a_{32}}
\cdots  \frac{a_{K1}}{a_{1K}}\right),
\end{equation}
which synthesizes the asymmetry of the coefficients $a_{ij}$ (see \cite{MR4020313} for more 
details). 

Our main result is the following theorem.
\begin{theorem}\label{thm:main_intro} 
Let $K\ge 3$, $a_{ij}>0$, $\omega\in\R$. Assume that $\mu<\pi^2$ and 
$(\varphi_1,\dots,\varphi_K)$ satisfies \eqref{eq:ass_fi_i}. There
exists 
\[
\bar s=(\bar s_1,\dots,\bar s_K)\in\R^K, 
\]
independent of $\mu$ and $\omega$, with $\bar s_i>0$ for all $i$, such that:
\begin{enumerate}
\item If $s=t\bar s$ for some $t>0$, then $\Scal_\rot$ contains an element with a unique singular point at $\mathbf{0}$. Moreover such element is unique and, denoting with 
$\Ucal$ a suitable linear combination of its components, we have
\begin{equation}\label{eq:expans_intro}
\Ucal(r\cos \vartheta,r\sin\vartheta)=A r^{\gamma} \cos\left(\frac{K}{2}\vartheta 
-\alpha \ln r\right)+ o (r^{\gamma})\qquad\text{as }r\to0,
\end{equation}
where
\[
\gamma=\frac{K}{2} + \frac{2\alpha^2}{K}\qquad \text{and}\qquad
0<A_0\le A(\bx)\le A_1.
\]
\item If $s\neq t\bar s$ for every $t>0$, then $\Scal_\rot$ contains no element with a unique singular point at $\mathbf{0}$.
\end{enumerate}
\end{theorem}
\begin{corollary}\label{coro:intro}
Under the assumptions of the above theorem, if the problem is invariant under a rotation of 
$2\pi/K$, i.e.\ 
\begin{equation}\label{eq:tutteuguali}
\varphi_{i+1}(\bx) = \varphi_1(\Rcal_{2\pi i/K}\bx)
\qquad\text{and}\qquad
\frac{a_{i(i+1)}}{a_{(i+1)i}}=\frac{a_{K1}}{a_{1K}},
\end{equation}
for every $i$, then 
\[
\bar s = (1,1,\dots,1).
\]
\end{corollary}
\begin{remark}
Notice that the asymptotic expansion \eqref{eq:expans_intro} implies that the free boundary, 
near the singular point $\mathbf{0}$, is the union of $K$ equi-distributed logarithmic spirals, as long as $\alpha\neq0$. On the other hand, in case $\alpha=0$, we obtain that the 
interfaces enter the origin with a definite angle. In particular, this holds true 
in the symmetric case $a_{ij}=a_{ji}$ for every $j\neq i$.
\end{remark}
\begin{remark}
A natural question concerns the dynamical stability of the solutions above. From this point of view, the study of the linearized problem of \eqref{eqn sys beta parab}, due to the presence of the large parameter $\beta$, does not seem a viable path. This leaves open the problem of stability, for the moment, although numerical simulations for \eqref{eqn sys beta parab}, 
with logistic reactions and $\beta$ large, suggest stability for some specific angular velocity $\omega$.
\end{remark}

We shall adopt a constructive point of view, building the solution by superposition of 
fundamental elementary modes. The dependence 
of such building blocks on the parameter $\omega$ and $\mu$ shows the presence of  resonances at exceptional values (see Section \ref{sec:special} for further details). 
As a byproduct of the analysis of resonances, we will prove the following results.

\begin{theorem}[Homogeneous boundary conditions]\label{prop:intr_res}
Let $K\ge 3$ and $a_{ij}>0$. If $(\mu,\omega)$ 
belongs to a suitable discrete set then there exists a nontrivial element of $\Scal_\rot$ with 
null traces. Analogous results hold for homogenous Neumann or Robin boundary conditions.
\end{theorem}
\begin{theorem}[Entire solutions]\label{prop:intr_entire}
Let $K\ge 3$ and $a_{ij}>0$. For almost every $(\mu,\omega)$ there exists an entire solution of 
\eqref{eqn Somega} in $\R^2$.
\end{theorem}

In the above results, the conditions on $(\mu,\omega)$ are explicit in terms of the zero set of 
suitable analytic functions in the complex plane. Indeed, in both cases, the solutions are 
explicit in terms of trigonometric and Bessel's functions. This allows us to study the  
structure of the free boundary of the entire solutions far away from the origin. It turns out 
that, at least when $\omega\neq0$, also at infinity the free boundary consists in 
equi-distributed spirals, now of arithmetic type. We refer to Lemma \ref{lem:asy_entire} and 
Remark \ref{rem:final} ahead for further details.

\begin{remark}\label{rem_entire_intro}
In the particular case $\alpha=\mu=0$, we obtain that the entire solution found in Theorem 
\ref{prop:intr_entire} is related to the nodal components of a smooth rotating solution of 
the pure heat equation. Let $\omega>0$, $k\ge1$ be an integer and let $I_k$ denote the modified Bessel 
function of the first kind, with parameter $k$. We have that the function
\[
	U(r e^{i\vartheta}, t) = \real \left[e^{ik (\vartheta+\omega t)} I_{k}\left(\frac{\sqrt{2\omega k}}{2} (1+i) r \right)\right]
\]
is an entire, eternal rotating solution of the heat equation
\[
U_t-\Delta U = 0\qquad \text{in }\R^2\times\R
\]
having $2k$ nodal regions, which coincide up to rotations multiple of $\pi/k$. The 
equi-distributed nodal lines admit a straight tangent as $r\to0$, while they behave 
like arithmetic spirals of equation $\vartheta=\sqrt{\frac{\omega}{2k}} r$ as $r\to+\infty$, see Fig.~\ref{fig:caloric}. Notice that, as $\omega\to0$, a suitable renormalization of $U$ 
converges to the entire harmonic function $\real z^k$.
\end{remark}
\begin{figure}[t]
	\begin{center}
		\includegraphics[width=.35\textwidth]
		{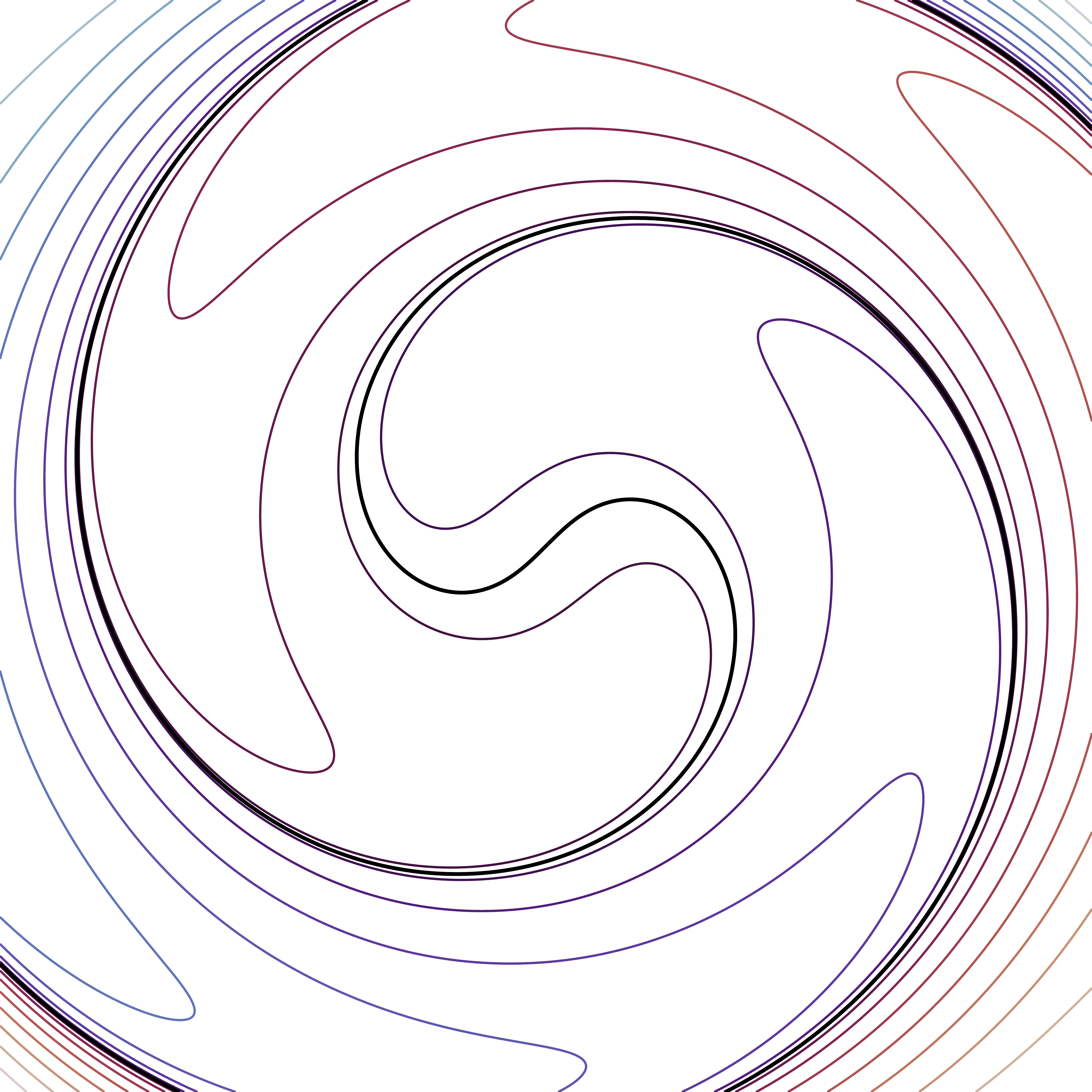} \hspace{2cm}
		\includegraphics[width=.35\textwidth]
		{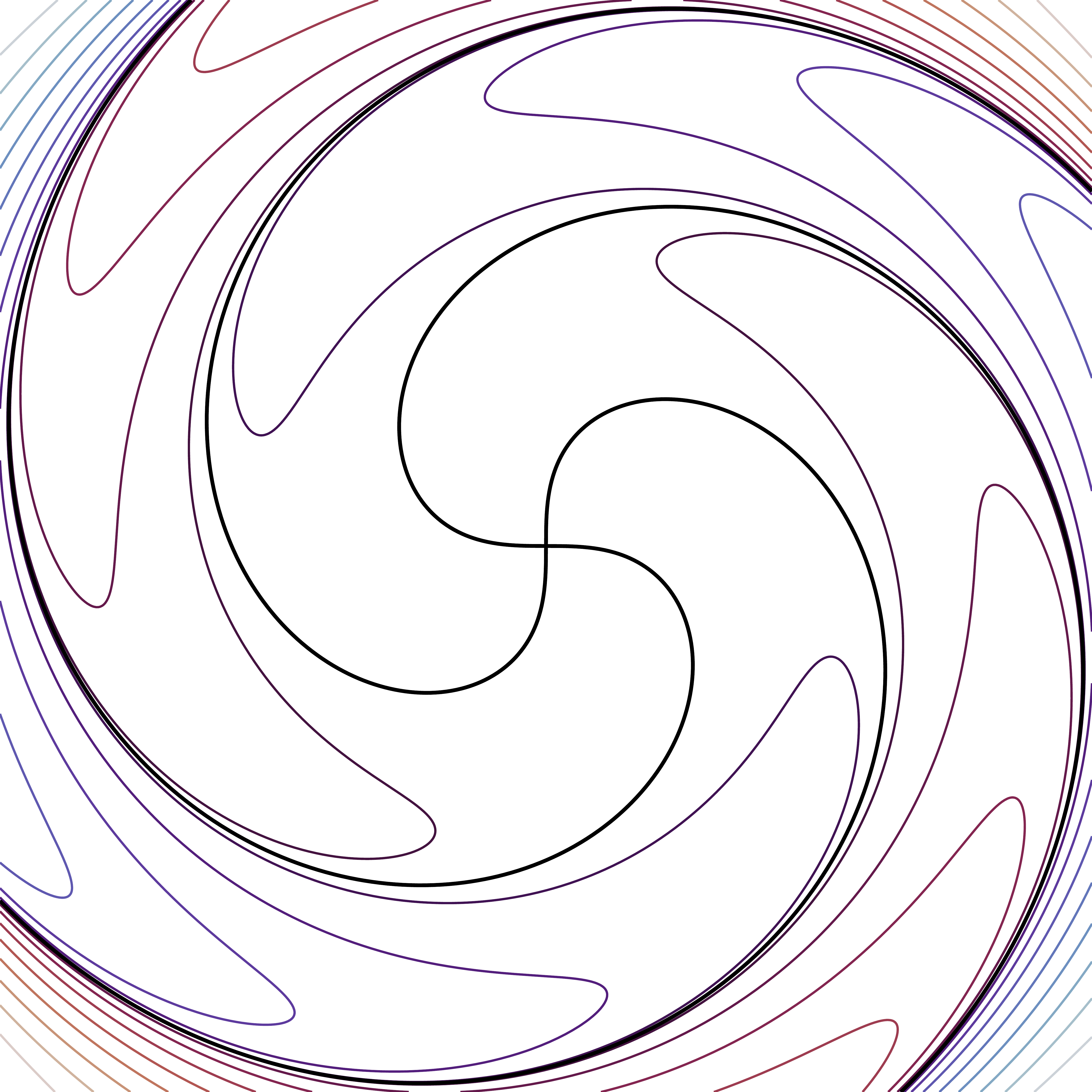}
		\caption{Contour lines of the rotating caloric functions in Remark \ref{rem_entire_intro}. Here $\omega = 1$,  $k =1$ and $k=2$, respectively. In black the nodal lines: the appearance of arithmetic spirals for $r$ large is rather clear in the picture.}
		\label{fig:caloric}
	\end{center}
\end{figure}
\begin{remark}
Notice that, by separation of variable, one may treat boundary value problems for rotating 
solutions also on other rotationally invariant domains $\Omega$, such as annuli or external 
domains. Of course, since in these cases $\mathbf{0}\not\in \Omega$, this cannot provide spiraling solutions, at least in our sense.
\end{remark}

Let us provide an explanation for our contruction. When a smooth curve 
separates two densities of an element of $\Scal_\rot$, at least locally, the 
gradients of the two densities are proportional across such interface. Indeed, by definition 
of $\hat u_i$, the function $a_{21}u_1 - a_{12}u_2$ solves an elliptic equation 
in a neighborhood of the interface.

Let us assume for 
concreteness $K=3$. In case the nodal structure of 
$(u_1,u_2,u_3)\in\Scal_\rot$ is the required one, as depicted in Fig.~\ref{fig:intro}, then a 
suitable linear combination of the components $u_i$ satisfies an equation on $B$, up to a curve. 
More precisely, let us define
\[
\Ucal = u_1 -\frac{a_{12}}{a_{21}}u_2 + \frac{a_{12}}{a_{21}}\cdot\frac{a_{23}}{a_{32}}u_3,
\qquad
\Gamma=\overline{\{u_1>0\}}\cap\overline{\{u_3>0\}}.
\] 
It is easy to check that
\[
-\Delta \Ucal + \omega \bx^\perp \cdot \nabla \Ucal = \mu \Ucal\text{ in }B\setminus\Gamma,
\]
while, if $\mathbf{0}\neq\bx_0\in\Gamma$ and $\alpha$ is defined as in \eqref{eq:def_alfa},
\[
\lim_{\substack{\bx\to\bx_0\\u_3(\bx)>0}} \nabla \Ucal(\bx) = -e^{2\pi\alpha}
\lim_{\substack{\bx\to\bx_0\\u_1(\bx)>0}} \nabla \Ucal(\bx).
\]
By composing with a conformal map between $B\setminus\{\mathbf{0}\}$ and its universal covering 
$\R\times (0,\infty)$, we can lift $\Ucal$ to a solution of a linear equation in the half-plane (see 
\eqref{eq:eq_in_the_halfplane} below), having 
a precise nodal structure. This connection is analyzed in Section \ref{sec:disk2halfplane}.

To prove Theorem \ref{thm:main_intro} we reverse the above argument: we start by solving 
the equation in the covering by separation of variables, in Section \ref{sec:sepvar}; next, 
we show in Section \ref{sec:nodalset} that, under suitable 
conditions, the solution has the appropriate nodal properties to be mapped back to the disk. In both 
these points, we have to deal with non-resonance/coerciveness conditions, leading to the 
assumption on $\mu$. On the other hand, the existence of the vector $\bar s$ is equivalent to the validity of suitable compatibility conditions, expressed in terms of the Fourier coefficients of the boundary data. Specifically, 
when $K=3$, $\bar s$ is any componentwise positive solution of the system:
\begin{equation}\label{eq:compX3}
\int_0^{2\pi} e^{-\alpha \vartheta}\Phi(\vartheta)\sin\left(\frac{\vartheta}{2}\right)\,dx = 
\int_0^{2\pi} e^{-\alpha \vartheta}\Phi(\vartheta)\cos\left(\frac{\vartheta}{2}\right)\,dx = 0,
\end{equation}
where 
\[
\Phi = s_1\varphi_1 -s_2\frac{a_{12}}{a_{21}}\varphi_2 + s_3\frac{a_{12}}{a_{21}}\cdot\frac{a_{23}}{a_{32}}\varphi_3.
\] 
We analyze the general compatibility conditions in Section \ref{sec:proofmain}, 
concluding the proof of Theorem \ref{thm:main_intro}. Finally, Theorems \ref{prop:intr_res} 
and \ref{prop:intr_entire} are proved in Section \ref{sec:special}.

\section{An equivalent problem in the half-plane}\label{sec:disk2halfplane}

As we mentioned, the proof of Theorems \ref{thm:main_intro}, \ref{prop:intr_res} 
and \ref{prop:intr_entire} is based on the connection between system \eqref{eqn Somega} 
and an equation in the half-plane, seen as the universal covering of the punctured disk. In
this section we analyze such connection.

Let $\mu,\omega$ be real parameters, and $v = v(x,y) \in 
C(\R\times[0,+\infty))$ be a classical solution of the equation
\begin{equation}\label{eq:eq_in_the_halfplane}
		-\Delta v + \omega e^{-2y} v_x = e^{-2y}\mu v \qquad x\in\R,\,y>0.
\end{equation}
In the following we assume that $v$ satisfies the following properties:
\begin{itemize}
\item[a)] there exists $\sigma\neq0$ such that 
\begin{equation}\label{eq:cond_a)}
v(x + 2\pi, y) = \sigma v(x,y),
\end{equation} 
for any  $x\in\R$, $y\ge 0$;
\item[b)] $v(x,y)= 0$ iff $(x,y)\in \overline{\cc}_i
\cap\overline{\cc}_{i+1} $ for some $i\in\Z$, where 
the non-empty nodal regions $\cc_i$ are open, connected, disjoint, unbounded and
	\[
    \begin{split}
    &\overline{\cc}_i\cap\{(x,0):x\in\R\}=\{(x,0):x_{i-1}\le x\le x_i\}\\
	&\overline{\cc}_i\cap\overline{\cc}_j\neq\emptyset \qquad\iff\qquad j-i=-1,0,1.
    \end{split}
	\]
	In particular, since $v$ is analytic for $y>0$, we obtain that 
	the set $\overline{\cc}_i\cap\overline{\cc}_{i+1} $ is actually a 
	locally analytic curve which accumulates both at $(x_i,0)$ and at 
	$y=\infty$;
\item[c)] $\left.v\right|_{\cc_i} \in H^1(\cc_i)$, for every $i\in \Z$ (or, equivalently, their trivial extensions belong to $H^1(\R\times(0,+\infty))$). 
\end{itemize}
We infer that $\bigcup_i\overline{\cc}_i = 
\R\times[0,+\infty)$, and that this covering is locally finite. Moreover, 
by a) the nodal set of $v$ is $2\pi$-periodic in the $x$-direction. Up to a 
translation, we can assume that $x_0=0$, 
so that in particular $v(0,0)=0$ and the number $K$ of nodal components, up to periodicity, can be defined as
\begin{equation}\label{eq:defk}
K=\#\{i: [x_{i-1},x_i]\subset[0,2\pi]\}\qquad \text{i.e.\ } S_{i+K} = S_{i}+(2\pi,0),
\forall i.
\end{equation}
Notice that $\sigma>0$ implies $K$ even, while $\sigma<0$ forces $K$ odd.

Finally, we introduce the following conformal map between the half-plane 
and the punctured disk:
\begin{equation}\label{eq:defT}
\mathcal{T} \from
\R \times (0,+\infty) \to  B\setminus\{\mathbf{0}\},\qquad
 \mathcal{T}\from (x,y) \mapsto \mathbf x=(e^{-y}\cos x,e^{-y}\sin x)
\end{equation}
(for more details about this map, see Remarks 2.17 and 2.19 in \cite{MR4020313}).

The main result of this section is the following.
\begin{proposition}\label{prop:back2disk}
Let $v\in  C(\R\times[0,+\infty))$ be a classical solution of
\eqref{eq:eq_in_the_halfplane}, satisfying a), b) and c), and let $K$
be defined as in \eqref{eq:defk}. Assume that the positive 
coefficients $a_{ij}$ and the parameter $\alpha$ satisfy
\begin{equation}\label{eq:rel_sigma_aij}
\prod_{i=1}^K \frac{a_{(i-1)i}}{a_{i(i-1)}}  = (-1)^K\sigma
\end{equation}
(understanding $a_{01}=a_{K1}$, $a_{10}=a_{1K}$).

For $i=1,\dots,K$ let us define
\begin{equation}\label{eq:def_li}
u_i = (-1)^{i+1} l_i \left.v\right|_{S_i} \circ \Tcal , \qquad
\text{with } l_1=1,\ l_{i} = \frac{a_{i(i-1)}}{a_{(i-1)i}}\cdot l_{i-1}
\end{equation}
(trivially extended in the whole $B$). Then $(u_1,\dots, u_K) \in \Scal_\rot$. Moreover, with respect to this $K$-tuple, the origin is the only point with higher multiplicity, with $m(\mathbf{0})=K$. 

Vice versa,  if $(u_1,\dots, u_K) \in \Scal_\rot$ has the origin as only singular point, 
then there exists $v$ such that the first part of the proposition holds.
\end{proposition}
\begin{remark}\label{rem:back2traces}
In case that the asymptotic behavior of the nodal zones $S_i$ is 
known, for $y\to+\infty$, then by composition with $\Tcal$ one can 
deduce the local description of the free boundary associated to  
$(u_1,\dots , u_K)$ near $\mathbf{0}$.
\end{remark}
\begin{proof}
By condition a) the functions $u_i$ are well-defined, by b) they satisfy $u_i\cdot u_j \equiv0$ as long as $j\neq i$, and by c) 
they belong to $H^1(B)$ (recall that $\Tcal$ is a conformal map). 
With direct computations one can check that 
\begin{equation}\label{eq:sub_str}
-\Delta u_i + \omega\, \bx^\perp \cdot \nabla u_i = \mu u_i
\qquad \text{ in }\omega_i:=\{u_i>0\}.  
\end{equation}
Analogously, using the definition of the coefficients $l_i$ (see eq.
\eqref{eq:def_li}), we have that 
\begin{equation}\label{eq:sup_str}
-\Delta \left(u_{i-1} - \frac{a_{(i-1)i}}{a_{i(i-1)}}u_i\right)  + \omega\, \bx^\perp \cdot \nabla \left(u_{i-1} - \frac{a_{(i-1)i}}{a_{i(i-1)}}u_i\right) = \mu \left(u_{i-1} - \frac{a_{(i-1)i}}{a_{i(i-1)}}u_i\right)  
\end{equation}
in the interior of $\overline{\omega}_{i-1} \cup 
\overline{\omega}_{i}$, $i=1,\dots, K$ (in case $i=1$ we keep 
understanding $i-1=K$, and the validity of \eqref{eq:sup_str} 
follows by \eqref{eq:rel_sigma_aij}). Notice that, when restricted 
to $\overline{\omega}_{i-1} \cup 
\overline{\omega}_{i}$, the function 
in \eqref{eq:sup_str} is a multiple of both $\hat u_{i-1}$ and 
$\hat u_i$.

We have to
show the validity of the inequalities
\begin{align}
\int_{B}\nabla u_i\cdot \nabla \varphi + \left[\omega\, \bx^\perp \cdot \nabla u_i - \mu u_i\right]\varphi \le 0,\label{eq:sub}\\
\int_{B}\nabla \hat u_i \cdot\nabla \varphi + \left[\omega\, \bx^\perp \cdot \nabla\hat u_i - \mu\hat u_i\right]\varphi \ge 0,\label{eq:sup}
\end{align}
for every Lipschitz, compactly supported, non-negative $\varphi$. 

First, let us consider any $\varphi$ such that  
$\varphi\equiv0$ in $B_\eps(\mathbf{0})$. Then \eqref{eq:sub} follows by integration by parts, since
\[
\int_{B}\nabla u_i\cdot \nabla \varphi + \left[\omega\, \bx^\perp \cdot \nabla u_i - \mu u_i\right]\varphi  = 
\int_{\omega_i\setminus B_\eps}\nabla u_i\cdot \nabla \varphi + \left[\omega\, \bx^\perp \cdot \nabla u_i - \mu u_i\right]\varphi = 
\int_{\partial \omega_i} \partial_\nu u_i \varphi \le0,
\]
where we used the regularity of $\partial \omega_i$ away from 
$\mathbf{0}$, the equation for $u_i$ and the fact that 
$\partial_\nu u_i\le 0$ on $\partial\omega_i$. On the other hand, to prove  
\eqref{eq:sup}, since $\varphi\equiv0$ in $B_\eps(\mathbf{0})$ we can use a partition of unity argument and assume that $\supp(\varphi)$ intersects at most two 
adjacent nodal regions. In case none of them is $\omega_i$, then 
$\hat u_i = -c_1 u_j - c_2 u_{j+1}$, with $c_i>0$, and  
\eqref{eq:sup} follows by applying twice \eqref{eq:sub}, with $i=j,j+1$; if $\supp(\varphi) \subset \overline{\omega}_{i-1} \cup 
\overline{\omega}_{i} \setminus B_\eps$ then \eqref{eq:sup_str} 
yields
\[
\int_{B}\nabla \hat u_i\cdot \nabla \varphi + \left[\omega\, \bx^\perp \cdot \nabla \hat u_i - \mu u_i\right]\varphi  = 
\int_{\overline{\omega}_{i-1} \cap 
\overline{\omega}_{i} \setminus B_\eps}\nabla \hat u_i\cdot \nabla \varphi + \left[\omega\, \bx^\perp \cdot \nabla \hat u_i - \mu \hat u_i\right]\varphi = 0,
\]
and the same holds true if if $\supp(\varphi) \subset \overline{\omega}_{i} \cup 
\overline{\omega}_{i+1} \setminus B_\eps$. 

Finally, let us consider any $\varphi$. We show how to prove 
\eqref{eq:sub}, \eqref{eq:sup} is analogous. For any $\eps > 0$ small, 
we define the function
	\[
		\eta(\mathbf x)=\begin{cases}
			0 & \mathbf x \in B_\eps\\
			\frac{|\mathbf x|-\eps}{\eps} & \mathbf x \in B_{2\eps} 
			\setminus B_\eps\\
			1  & \mathbf x \in B \setminus B_{2\eps}.
		\end{cases}
	\]
	Then $\varphi \eta = 0$ in $B_\eps$ and by the previous part
	\[
		\int_B ( \nabla u_i \cdot \nabla \varphi) \eta + \int_B (\nabla u_i \cdot \nabla \eta ) \varphi  + \int_B \left[\omega\, \bx^\perp \cdot \nabla u_i - \mu u_i\right]\eta\varphi 
		\leq 0.
	\]
	Since $\varphi$ is Lipschitz, we have
	\[
		 \left|\int_B (\nabla u_i \cdot \nabla \eta ) \varphi \right| \leq \frac{1}{\eps} \int_{B_{2\eps}\setminus B_\eps}  \left|\nabla u_i \right| \varphi \leq \frac{1}{\eps} \| u_i \|_{H^1(B_{2\eps})} \|\varphi \|_{L^2(B_{2\eps})} \leq C  \| u_i \|_{H^1(B_{2\eps})} \|\varphi\|_{L^\infty}.
 	\]
 	Thus we find the estimate
 	\[
 		\int_B ( \nabla u_i \cdot \nabla \varphi) \eta + \int_B \left[\omega\, \bx^\perp \cdot \nabla u_i - \mu u_i\right]\eta\varphi \leq  C \| u_i \|_{H^1(B_{2\eps})} \|\varphi\|_{L^\infty}.
 	\]
 	Taking the limit for $\eps \to 0$, since $\eta$ converges monotonically to $1$, we conclude that
 	\[
 		\int_B \nabla u_i \cdot \nabla \varphi + \left[\omega\, \bx^\perp \cdot \nabla u_i - \mu u_i\right] \varphi \leq 0 ,
 	\]
concluding the proof of the first assertion.

The second part follows by defining
\begin{equation}\label{eq:def_v}
v\circ\Tcal = \sum_{i=1}^K  \frac{(-1)^{i+1}}{l_i} u_i,
\end{equation}
and then deriving $v$ by a lifting argument. We refer to \cite[Sect. 2]{MR4020313} for 
further details.
\end{proof}

\section{Solutions in the half-plane}\label{sec:sepvar}

Let $\mu,\alpha,\omega\in\R$. Given the trace
\[
\Phi\from[0,2\pi]\to\R,\qquad \Phi(0) = \Phi(2\pi) =0,
\]
we look for solutions $v$ of the problem in the half-plane:
\begin{equation}\label{eq:linear_in_the_halfplane}
	\begin{cases}
		-\Delta v + \omega e^{-2y} v_x = e^{-2y}\mu v & x\in\R,\,y>0\\
	    v(x + 2\pi, y) = e^{2\pi\alpha} v(x,y)              & x\in\R,\,y\ge 0\\
		v(x,0)=\Phi(x) & 0\le x \le 2\pi.
	\end{cases}
\end{equation}

Notice that we are considering equation \eqref{eq:eq_in_the_halfplane} 
together with condition \eqref{eq:cond_a)}, in the case $\sigma = 
e^{2\pi\alpha}>0$ (recall the definition \eqref{eq:def_alfa} and the relation \eqref{eq:rel_sigma_aij}). As we noticed, this entails an even number of nodal 
zones, in the period. One can easily modify our arguments to deal with 
an odd one, i.e.\ with $\sigma<0$, for instance with the change of variables 
$(x,y)\mapsto(x/2,y/2)$, $\sigma\mapsto\sigma^2$. In a completely equivalent way, 
one can work with $2\pi$-periodicity and take $\alpha = \frac{1}{2\pi} \ln |\sigma| + \frac{i}{2}\in\C$.

To solve \eqref{eq:linear_in_the_halfplane}, we first transform it into a periodic problem, and
then use separation of variables to write the solution in Fourier series. To this aim, we notice that $v$ solves \eqref{eq:linear_in_the_halfplane} if and only if
\[
	w(x,y) := e^{-\alpha x} v(x,y)
\]
solves
\begin{equation}\label{eq:linear_in_the_halfplane_w}
	\begin{cases}
		-\Delta w + (\omega e^{-2y} - 2 \alpha) w_x +  [(\alpha\omega - \mu)e^{-2y} - \alpha^2] w =  0 & x\in\R,\,y>0\\
	    w(x + 2\pi, y) = w(x,y)              & x\in\R,\,y\ge 0\\
		w(x,0)= e^{-\alpha x} \Phi(x) & 0\le x \le 2\pi.
	\end{cases}
\end{equation}
Of course, if $\alpha=0$ then $v$ and $w$ coincide. Either way, with a little abuse of notation,
we can extend $\Phi$ to $\R$ in such a way that $e^{-\alpha x} \Phi(x)$ is $2\pi$-periodic. At
least formally, we can expand $w$ in Fourier series, and write
\[
	w(x,y) = \sum_{k \in \Z} W_{k}(y) e^{i k x}.
\]
Plugging this expression into \eqref{eq:linear_in_the_halfplane_w} we obtain that
the coefficients $W_k:\overline{\R^+} \to \C$, $k \in \Z$ must solve the ordinary
differential equation
\begin{equation} \label{eq.w}
	W_k''(y)  = \left[(k-i \alpha)^2 + \left( \omega \alpha - \mu +i \omega k \right)
	e^{-2y} \right] W_k(y), \qquad y > 0.
\end{equation}
We can solve boundary value problems associated with \eqref{eq.w} by using the Fredholm Alternative and the Lax-Milgram Theorem, settled in complex Hilbert spaces. We are looking
for solutions of \eqref{eq:linear_in_the_halfplane_w} that change sign for $y\to+\infty$. As we will see in Lemma \ref{lem H1 and nodal}, this entails that the term corresponding to  $k=0$ in the expansion should not be present. For this reason we consider $k\neq0$ from now on.
\begin{lemma} \label{lem:weak_coerc}
For any $k \in \Z \setminus\{0\}$, $\alpha \in \R$, there exists a sequence
$\{\lambda_{n}\}_{n\in \N} \subset \C$, with $|\lambda_n| \to +\infty$ as
$n\to +\infty$, such that the problem
\begin{equation} \label{eq.w.2}
	\begin{cases}
		X_k''(y) = \left[(k-i \alpha)^2 + \left( \omega \alpha - \mu +i \omega k \right)
	e^{-2y} \right]  X_k(y), &y > 0,\\
		X_k(0) = 1, \qquad X_k \in H^1(\R^+;\C),
	\end{cases}
\end{equation}
admits a unique solution if and only if
\begin{equation}\label{eq:cond_res}
	\omega \alpha - \mu + i \omega k \not \in \{ \lambda_n\}_{n \in \N},
\end{equation}
while no solution exists in the complementary case.
\end{lemma}
\begin{proof}
We shall consider the case $k \geq 1$, as the case $k \leq -1$ follows by the same arguments, up to the change of sign
\[
	(\alpha, \omega, \mu, k) \mapsto (-\alpha, -\omega, \mu, -k).
\]
In particular one can verify that $X_{-k}(y) = \overline{X_k(y)}$ for any $k \in \Z$ and $y\geq0$ 
(in case one of them exists).
We proceed through several steps.

\textbf{Step 1. Weak formulation of the problem.}
Letting $X_k = U + U_0$, where $U_0:=e^{-(k-i\alpha)y}$, we are led to find, if it exists, a function $U \in  H^1_0(\R^+; \C)$ solution of
\[
-U'' + \left[(k-i \alpha)^2 + \left( \omega \alpha - \mu +i \omega k \right)
	e^{-2y} \right]  U = - \left( \omega \alpha - \mu +i \omega k \right)
	e^{-2y}e^{-(k-i\alpha)y},\qquad y > 0.
\]
We settle the problem in the space
\[
	H = H_0^1(\R^+; \C), \qquad \|u\|_H^2 = \int_0^\infty |U'|^2 + |U|^2.
\]
To proceed, we introduce the sesquilinear forms $a_R, a_I$ as
\[
    a_R(U,V) = \int_0^\infty U' \bar V' + [(k^2-\alpha^2) + (\omega \alpha - \mu)e^{-2y}]U \bar{V}, \quad a_I(U,V) = \int_0^\infty (-2\alpha k + \omega k e^{-2y}) U \bar{V}
\]
and the antilinear form $l$ as
\begin{equation}\label{eq:l}
    l(V) = -(\omega \alpha -\mu + i \omega k) \int_0^\infty e^{-2y}U_0 \bar V=-(\omega \alpha -\mu + i \omega k) \int_0^\infty e^{-(k+2-i\alpha)y} \bar V.
\end{equation}
In this way, we are reduced to solve the following variational problem: finding $U \in  H$ such that
\begin{equation}\label{eq:PVA}
a(U,V) = a_R(U,V) + i a_I(U,V) = l(V) \qquad \forall V \in  H.
\end{equation}
Notice that  both $a$ and $l$ are continuous: indeed, since $|e^{-2y}|\le1$ for $y\ge0$, it is easy to see that
\[
    |a(U,V)| \leq \left(k^2 + \alpha^2 + \sqrt{(\omega \alpha-\mu)^2+(\omega k)^2}\right) \|u\|_H \|v\|_H.
\]
Similarly, for $l$ we obtain
\[
    |l(V)| \leq |(\omega \alpha -\mu + i \omega k)| \int_0^\infty e^{-(k+2)y} | V |  \leq \frac{\sqrt{(\omega \alpha-\mu)^2+(\omega k)^2}}{\sqrt{2(k+2)}} \left(\int_0^\infty  | V |^2 \right)^{1/2}.
\]
For future purposes we notice that, for every $U\in H$, both $a_R(U,U)$ and $a_I(U,U)$ are real
numbers: indeed, $a_R(U,U)$ and $a_I(U,U)$ are, respectively, the real and imaginary part of
$a(U,U)$. We can exploit the Cauchy-Schwarz inequality (for real $2$-dimensional vectors) to
find that
\begin{equation}\label{eq:stima_precoerc}
\begin{split}
\left|a(U,U)\right| &= \sup_{K \in \R} \frac{a_R(U,U) + K a_I(U,U)}{\sqrt{1+K^2}} \ge
\frac{k}{\sqrt{\alpha^2 + k^2}}\left( a_R(U,U) - \frac{\alpha}{k} a_I(U,U) \right)\\
&=   \frac{k}{\sqrt{\alpha^2 + k^2}}\int_0^\infty \left[|U'|^2 +  \left(k^2 + \alpha^2\right) |U|^2\right] -  \frac{k\mu}{\sqrt{\alpha^2 + k^2}}\int_0^\infty e^{-2y} |U|^2.
\end{split}
\end{equation}
In order to prove existence and uniqueness of a solution $U$ we shall make use of the classical Fredholm alternative theorem. In particular, we shall find that \eqref{eq:PVA} admits a unique solution $U \in H^1_0(\R^+; \C)$ if and only if $0$ is not an eigenvalue of $a$ (more precisely,
and equivalently, $0$ is not an eigenvalue of the conjugate transpose sequilinear form $a^\dag$).

\textbf{Step 2. A related eigenvalue problem.}
To proceed, we introduce the (adjoint) eigenvalue problem: finding $\lambda\in \C$ and
$V\in H\setminus\{0\}$ such that
\[
\int_0^\infty \left[U' \bar V' + (k-i\alpha)^2 U \bar{V}\right] + \lambda \int_0^\infty  e^{-2y}U \bar{V}
=0 \qquad \forall U \in  H.
\]
Defining the weighted space
\[
	L = \left\{U \in L^1_\loc(\R^+;\C) : \|U\|_L^2 = \int_0^\infty e^{-2y} |U|^2 < +\infty \right\},
\]
we have that $H \subset L = L^* \subset H^*$ is a Hilbert triplet, with $H$
compactly embedded in $L$ (see Lemma \ref{lem comp emb}). Then standard spectral
theory (see e.g.\ \cite[Ch.\ 3, Thm.\ 6.26]{MR1335452}) yields the existence of a
sequence of eigenvalues $\{\lambda_{n}\}_{n\in \N} \subset \C$, with $|\lambda_n| \to +\infty$, and it is straightforward to show that $V\neq0$ satisfies
\begin{equation}\label{eq:fredh}
a(U,V)= 0 \quad \forall U \in  H
\quad \iff \quad
\begin{array}{c}
\omega \alpha - \mu + i \omega k = \lambda_n,\\
\text{and $V=V_n$ is an associated eigenfunction.}
\end{array}
\end{equation}
Notice that each $\lambda_n$ is a simple eigenvalue, by uniqueness of the Cauchy problem for ODEs.

\textbf{Step 3. Application of the Babu\v{s}ka-Lax-Milgram theorem.} To conclude the
invertible case, we show that, if $\omega \alpha - \mu + i \omega k \neq
\lambda_n$, for every $n$, then there exists a unique solution to \eqref{eq:PVA}. To this
aim, we apply a generalization of the Lax-Milgram theorem due to Babu\v{s}ka
\cite[Thm. 2.1]{MR288971} (with $H_1=H_2=H$). After the previous steps, in order to
apply such result to \eqref{eq:PVA}, we only need to show that, if $\omega \alpha - \mu + i \omega k \neq \lambda_n$ for every $n$, then the following inf-sup conditions hold:
\[
\inf_{\|V\|_H=1} \sup_{\|U\|_H=1} |a(U,V)|\ge C_2>0,\qquad
\inf_{\|U\|_H=1} \sup_{\|V\|_H=1} |a(U,V)|\ge C_3>0,
\]
for suitable constants $C_2,C_3$.
We prove the first inequality, the second one being analogous. Assume by contradiction that the sequence $\{V_n\}_n$ satisfies
\[
\|V_n\|_H=1,\qquad\qquad  |a(U,V_n)|\le \frac{1}{n} \|U\|_H \quad \forall U \in  H.
\]
In particular, as $n\to+\infty$, $a(V_n,V_n)\to0$. Moreover, up to subsequences, $V_n$ converges to $V_\infty$, both weakly in $H$ and strongly
in $L$ (by compact embedding). Thus $a(U,V_\infty) = 0$ for every $U\in H$.
Since $\omega \alpha - \mu + i \omega k \neq \lambda_n$, for every $n$, and recalling
\eqref{eq:fredh}, we deduce that $V_\infty \equiv 0$. Since $k^2 \ge 1$,  \eqref{eq:stima_precoerc} yields
\[
\begin{split}
o(1) = |a(V_n,V_n)| \ge \frac{k}{\sqrt{\alpha^2 + k^2}} \|V_n\|^2_H -
\frac{k\mu}{\sqrt{\alpha^2 + k^2}} \|V_n\|^2_L = \frac{k}{\sqrt{\alpha^2 + k^2}} + o(1)
\end{split}
\]
as $n\to\infty$, a contradiction.

\textbf{Step 4. Non-existence in the resonant case.} Finally, assume that $\omega \alpha - \mu + i \omega k = \lambda_n$, for some $n$, and let $V_n\not\equiv0$ be an associated eigenfunction of the adjoint
problem:
\[
a(U,V_n) = \int_0^\infty \left[U' \bar V_n' + (k-i\alpha)^2 U \bar{V}_n\right] + \lambda_n \int_0^\infty  e^{-2y}U \bar{V}_n
=0  \qquad \forall U \in  H.
\]
This forces
\begin{equation}\label{eq:fredh_NO}
-\bar V_n'' + (k-i\alpha)^2 \bar V_n + \lambda_n e^{-2y} \bar{V}_n = 0\qquad\text{on }(0,\infty);
\end{equation}
in particular, $V_n\in H^2(0,+\infty)$, and thus $V'_n(y)\to0$ as $y\to+\infty$. Moreover,
by uniqueness of the Cauchy problem, $V'_n(0) \neq 0$.

In the case we are considering \eqref{eq:PVA} rewrites as
\[
a(U,V) = (-\lambda_n U_0,V)_L \qquad \forall V \in  H,
\]
where $U_0 = e^{-(k-i\alpha)y}$. By Fredholm's alternative, in this case \eqref{eq:PVA} is
solvable if and only if the compatibility condition
\[
(-\lambda_n U_0,V_n)_L =0
\]
holds true. Using \eqref{eq:fredh_NO} we have
\[
(-\lambda_n U_0,V_n)_L = - \lambda_n  \int_0^\infty  e^{-2y}U_0 \bar{V}_n =
U(0)\bar V'_n(0) + \int_0^\infty \left[U'_0 \bar V'_n + (k-i\alpha)^2 U_0 \bar{V}_n\right]
=\bar V'_n(0) \neq 0.\qedhere
\]
\end{proof}
The resonance set in the previous lemma can be characterized in terms of the zero set of
the following function $\modB_\nu$, depending on the complex parameter $\nu$:
\begin{equation}\label{eq:fintaBessel}
\modB_\nu(z) =   \sum_{n=0}^\infty \frac{1}{n! \Gamma (n+1+\nu)} \left( \frac{z}{4}\right)^{n}.
\end{equation}
Notice that, for any $\nu\in\C$, $\modB_\nu$ is analytic on
$\C$ (recall that $\Gamma$ has no zeroes, but only simple poles at each non-positive integer
$-k$: in such case, we understand $1/\Gamma(-k) = 0$). As a matter of fact, $\modB_\nu$ is related to $I_\nu$, the modified Bessel function of the first kind, with parameter $\nu\in\C$, by the formula
\begin{equation}\label{eq modB and I}
I_\nu(z) =  \left( \frac{z}{2}\right)^{\nu} \modB_\nu(z^2)
\end{equation}
(in turn, $I_\nu(z) = e^{-i\nu\pi/2} J_\nu(iz)$, where $J_\nu$ is the usual Bessel function of the
first kind). Notice that, in case $\nu\not\in\Z$, $I_\nu$ is a multivalued function because of
the complex exponentiation $z^\nu$. Nonetheless, the zero set of (any determination of) $I_\nu$
coincides with  the complex square root of the zero set of $\modB_\nu$, with the exception for $0$.
\begin{lemma} \label{lem:bessel}
For any $k \in \Z \setminus\{0\}$, $\alpha \in \R$, let $\{\lambda_{n}\}_{n\in \N} \subset \C$
denote the sequence defined in Lemma \ref{lem:weak_coerc}. Then
\[
\{\lambda_{n}\}_{n\in \N} = \{z\in\C\setminus\{0\} : \modB_{\sign(k)(k-i\alpha)}(z)=0\},
\]
where $\modB_\nu$ is defined in \eqref{eq:fintaBessel} for every $\nu\in\C$.

Moreover, whenever
$\lambda := \omega \alpha - \mu + i \omega k \not \in \{ \lambda_n\}_{n \in \N}$, the unique
solution of \eqref{eq.w.2} is
\[
X_k(y) =  \frac{\modB_{\nu}(\lambda e^{-2y})}{\modB_{\nu}(\lambda)}
e^{-\nu y}
\]
($X_k(y) = e^{-\nu y}$ in case $\lambda = 0$), where $\nu = \sign(k)(k-i\alpha)$ whenever
$k\neq 0$.
\end{lemma}
Equivalently, we could write
\[
	X_k(y) = \frac{I_\nu\left(\sqrt{\lambda}e^{-y}\right)}{I_\nu\left(\sqrt{\lambda}\right)}
\]
and such identity is not ambiguous as long as we choose the same determinations both in the 
numerator and in the denominator.
\begin{proof}
Again, we treat the case $k\ge1$, the case $k\le-1$ following with minor changes.
With the above notation
\[
\nu = k - i\alpha, \qquad \lambda = \omega \alpha - \mu +i \omega k,
\]
the second order linear ODE in \eqref{eq.w.2} writes
\begin{equation}\label{eqn xab}
	x''(y) = \left[ \nu^2 + \lambda e^{-2y} \right]  x(y) .
\end{equation}
We assume $\lambda \neq 0$, the complementary case being trivial. Let us consider the functions $x_{\pm\nu}(y)$ defined as
\[
x_{\pm\nu}(y) = \modB_{\pm\nu} (\lambda e^{-2y})e^{\mp\nu y} =
\sum_{n \geq 0} c_{\pm\nu,n} e^{(-2n\mp\nu) y}, \qquad \text{ where }c_{\pm\nu,n} =
\frac{1}{n! \Gamma (n+1\pm\nu)} \left( \frac{\lambda}{4}\right)^{n}
\]
(again, we understand $c_{\pm\nu,n} =0$ whenever $-(n+1\pm\nu)\in\N$).
We notice that $4n(n\pm\nu) c_{\pm\nu,n} = \lambda c_{\pm\nu,n-1}$. Then
\[
\begin{split}
x''_{\pm\nu}(y) &=
\sum_{n \geq 0} (-2n\mp \nu)^2c_{\pm\nu,n} e^{(-2n\mp \nu) y} =
\sum_{n \geq 0} \nu^2 c_{\pm\nu,n} e^{(-2n\mp \nu) y}  + \sum_{n \geq 0} 4n(n\pm\nu) c_{\pm\nu,n} e^{(-2n\mp \nu) y} \\
&= \nu^2 x_{\pm\nu}(y) + \lambda \sum_{n \geq 1} c_{\pm\nu,n-1} e^{(-2n\mp \nu)y} = \left[ \nu^2 +
\lambda e^{-2y} \right]  x_{\pm\nu}(y),
\end{split}
\]
that is, both $x_{\pm\nu}$ solve the second order linear ODE \eqref{eqn xab}.

Let us first assume that $\alpha\neq 0$. Then $-(n+1\pm\nu)\not\in\N$, for every $n$, and we obtain that $\modB_{\pm\nu} (\lambda e^{-2y}) = 1/\Gamma(1\pm\nu) + o (1)$, that is,
\[
x_{\pm\nu}(y) = \frac{1}{\Gamma(1\pm\nu)} e^{\mp \nu y} + o(e^{\mp \nu y})
\qquad \text{as }y\to+\infty.
\]
Then $x_{\pm\nu}$ are linearly independent, and any solution of \eqref{eqn xab} is of the form
\[
x(y) = C_+ x_\nu(y) + C_- x_{-\nu}(y),\qquad C_\pm\in\C.
\]
Since $\nu = k - i\alpha$ and $k\ge1$, we have that $x\in H^1(0,+\infty)$ if and only if
$C_-=0$. As a consequence, \eqref{eq.w.2} is (uniquely) solvable if and only if $x_\nu(0) = \modB_{\nu} (\lambda) \neq 0$, and the lemma follows.

On the other hand, let $\alpha =0$ (and $\lambda \neq 0$).
In this case $\nu = k \ge 1$, and
\[
c_{-k,n+k}= \frac{1}{(n+k)!n!} \left( \frac{\lambda}{4}\right)^{n+k} =
\left( \frac{\lambda}{4}\right)^{k} c_{k,n},
\]
for every $n\ge0$, therefore the functions $x_{\pm k}$ are no longer linearly independent.
By differentiating \eqref{eqn xab} with respect to $\nu$, one can easily see that a second
independent solution of \eqref{eqn xab} can be obtained as
\[
\tilde x_{k}  = \left[\left( \frac{\lambda}{4}\right)^{k}\frac{\partial x_{\nu}}{\partial\nu} -\frac{\partial x_{-\nu}}{\partial\nu}\right]_{\nu=k},
\]
mimicking the procedure that leads to the (modified) Bessel functions of the second kind.
Since  $\Gamma (n+1-k)$ has a simple pole at $n=0$, we have that
\[
\lim_{\nu\to k} \frac{\partial c_{-\nu,0}}{\partial\nu} = (-1)^k(k-1)!
\qquad
\text{and}
\qquad
\tilde x_{k}(y) = (-1)^k(k-1)!e^{k y} + o(e^{k y})
\text{ as }y\to+\infty
\]
(see \cite[Sec. 7.2.5, p. 9]{MR0058756vol2} for more details). Thus also
in this case $\tilde x_{k}\not \in H^1(0,+\infty)$, and the lemma follows.
\end{proof}
\begin{corollary}\label{coro:decad}
Let $X_k$ denote the solution of \eqref{eq.w.2}. Then, for some $C\neq0$,
\[
X_k(y) = C e^{-\sign(k) (k-i\alpha)y} + O\left(e^{-(|k|+2)y}\right)
\qquad\text{as }y\to+\infty.
\]
\end{corollary}
\begin{remark}\label{rem:simpleeig}
As a byproduct of the proof of Lemma \ref{lem:bessel} we have that the eigenvalues 
$\lambda_n$ are all simple in $H^1_0(\R^+;\C)$. Indeed, the general solution 
of the corresponding eigenequation is a $2$-dimensional vector space of complex 
valued functions, but only a $1$-dimensional subspace consists of $H^1$ functions, 
of the form 
\[
C\modB_{\sign(k)(k-i\alpha)}(\lambda_n e^{-2y})e^{-\sign(k)(k-i\alpha) y},\qquad C\in\C.
\]
\end{remark}
In view of writing $w$ as a series in terms of the solutions $X_k$, we need to estimate the
asymptotic behaviors as $k \to \infty$ of their $L^2$ and $H^1$ norms.
\begin{lemma}\label{lem est xk}
Let $\alpha,\mu,\omega$ be fixed, in such a way that \eqref{eq:cond_res} holds true for every $k\neq0$. Then $X_k$ satisfies
\begin{equation}\label{eq:decad}
\left( \int_0^\infty  |X_k|^2\right)^{1/2} \le \frac{C}{\sqrt{|k|}},  \qquad \left( \int_0^\infty  |X_k'|^2\right)^{1/2} \le C\sqrt{|k|} \qquad \text{and} \qquad \| X_k \|_{L^\infty(0,+\infty)} \leq \sqrt{2} C,
\end{equation}
where $C$ depends only on $\alpha,\mu,\omega$.
\end{lemma}
\begin{proof}
As usual, for concreteness we assume $k\ge1$. As in the proof of Lemma \ref{lem:weak_coerc} we write
$X_k = U + e^{-(k-i\alpha)y}$.
In order to prove \eqref{eq:decad} we distinguish between two cases, corresponding to the instances $k$ small and $k$ large. Indeed for any fixed $\bar k$, which we will choose later in terms of $\alpha, \mu, \omega$, the estimate \eqref{eq:decad} is true for $k<\bar k$ and a suitable constant $C$. Next, for
$k\ge \bar k$, we estimate the norms of $U$ using the identity
\[
    |a(U,U)| = |l(U)|.
\]
Recalling \eqref{eq:l}, we have
\[
|l(U)| \le |\omega\alpha-\mu+i\omega k| \int_0^\infty |e^{-(k+2)y}| |U|
\leq \frac{\sqrt{(\omega \alpha-\mu)^2+(\omega k)^2}}{\sqrt{2(k+2)}} \left( \int_0^\infty  |U|^2  \right)^{1/2}
.
\]
Using \eqref{eq:stima_precoerc} we obtain
\[
\frac{k}{\sqrt{k^2+\alpha^2 }}\int_0^\infty \left[|U'|^2 +  \left(k^2 + \alpha^2
-\mu^+\right) |U|^2 \right]\leq \frac{\sqrt{(\omega \alpha-\mu)^2+(\omega k)^2}}{\sqrt{2(k+2)}} \left( \int_0^\infty  |U|^2  \right)^{1/2}.
\]
Then
\[
    \left( \int_0^\infty  |U|^2  \right)^{1/2} \leq
    \frac{\sqrt{k^2+\alpha^2}}{k\left(k^2 + \alpha^2 - \mu^+ \right)}\cdot \frac{\sqrt{(\omega \alpha-\mu)^2+(\omega k)^2}}{\sqrt{2(k+2)}} \leq \frac{|\omega|}{k^{3/2}},
\]
whence
\[
    \left( \int_0^\infty  |U'|^2\right)^{1/2} \leq \left[  \frac{\sqrt{k^2+\alpha^2}}{k} \frac{\sqrt{(\omega \alpha-\mu)^2+(\omega k)^2}}{\sqrt{2(k+2)}} \left(\int_0^\infty  | U |^2 \right)  \right]^{1/2}  \leq \frac{|\omega|^{3/2}}{k^{5/4}},
\]
for $k\ge \bar k$ sufficiently large (depending on $\omega, \mu,\alpha$).

Coming back to $X_k = U + e^{-(k-i\alpha)y}$, we finally obtain
\[
    \left( \int_0^\infty  |X_k|^2\right)^{1/2} \leq \left( \int_0^\infty  |U|^2\right)^{1/2}
    + \left( \int_0^\infty  e^{-2ky}\right)^{1/2}\le \frac{|\omega|}{k^{3/2}} + \frac{1}{\sqrt{2k}} \le \frac{1}{\sqrt{k}}
\]
and
\[
    \left( \int_0^\infty  |X_k'|^2\right)^{1/2} \leq \left( \int_0^\infty  |U'|^2\right)^{1/2}
    +\left( \int_0^\infty  |k-i\alpha|^2 e^{-2ky}\right)^{1/2}  \le \frac{|\omega|^{3/2}}{k^{5/4}} + \sqrt{\frac{k^2+\alpha^2}{2k}} \le \sqrt{k} ,
\]
for $k$ sufficiently large (depending on $\omega, \mu,\alpha$), concluding the $H^1$ estimates. 
Finally, by Corollary \eqref{coro:decad}, for any $y > 0$
\[
	X_k(y)^2 = -\int_y^\infty 2 X_k(t) X_k'(t) dt \leq 2 \left( \int_0^\infty  |X_k|^2\right)^{1/2} \left( \int_0^\infty  |X_k'|^2\right)^{1/2} \leq 2 C^2
\]
and the last estimate follows.
\end{proof}
Next we provide explicit sufficient conditions for the validity of condition
\eqref{eq:cond_res}.
\begin{lemma}\label{lem:strong_coercivity}
A sufficient condition for \eqref{eq:cond_res} to hold true is that
\begin{equation}\label{eqn ex coerc}
    \sup\left\{ \left(j_{\tau, 1} \right)^2 -\frac{\omega}{2\alpha} \tau^2  :  \tau > 0 \right\} > \mu - \frac{\omega}{2\alpha} \left( k^2+\alpha^2 \right),
\end{equation}
where $j_{\tau , 1}$ denotes the first (positive) zero of the standard Bessel function of first kind of order $\tau>0$.

This is the case, for instance, if
\begin{equation}\label{eq:_altre_coerc}
\text{either }\quad \mu < \left(j_{0 , 1}+\sqrt{k^2 + \alpha^2}\right)^2, \qquad
\text{or }\quad 	\frac{\omega}{\alpha} < 2.
\end{equation}
In particular, for any choice of $\alpha,\omega,\mu$, if $|k|$ is sufficiently large then \eqref{eq:cond_res} holds true.
\end{lemma}
\begin{proof}
Using the notation introduced in the proof of Lemma \ref{lem:weak_coerc}, we are going to show that, under
the present assumptions, the sesquilinear form $a$ is coercive. By the first estimate in \eqref{eq:stima_precoerc},  this follows once we find $K\in\R$ such that the quadratic form (with real coefficients)
\[
a_R(U,U) + a_I(U,U) K = \int_0^\infty |U'|^2 +  \left(k^2 - \alpha^2 - 2\alpha k K \right) |U|^2  + \left( (\omega \alpha - \mu) + \omega k K \right)e^{-2y} |U|^2
\]
is strictly positive. To this aim it is not difficult to check that we have to ask
$k^2 - \alpha^2 - 2\alpha k K >0$. For this reason it is convenient to introduce the
parameters $\tau > 0$, $b=b(\tau)$, such that
\[
    K = \frac{k^2 - \alpha^2 - \tau^2}{2\alpha k}, \qquad b = -\left( (\omega \alpha - \mu) + \omega k K \right)=  \mu +
    \frac{\omega}{2\alpha} \left( \tau^2 - (k^2+\alpha^2)\right).
\]
In this way we are reduced to find $\tau>0$ such that the quadratic form
\[
U\mapsto\int_0^\infty |U'|^2 +  (\tau^2 - b e^{-2y}) |U|^2
\]
is strictly positive. This quadratic form can be studied by standard arguments, we postpone the details to Lemma \ref{lem coer} in the Appendix. We obtain that it is coercive if and only if
\[
   b= \mu + \frac{\omega}{2\alpha} \left( \tau^2 - (k^2+\alpha^2)\right) <  \left(j_{\tau , 1}\right)^2,
\]
and \eqref{eqn ex coerc} follows. In order to make this condition more explicit, we exploit the fact that
\[
j_{\tau , 1} \ge j_{0 , 1}+\tau,\qquad\text{for every }\tau\ge0
\]
(see \cite{mccann_love_1982}). Therefore, a stronger condition than \eqref{eqn ex coerc} is
\[
\mu + \frac{\omega}{2\alpha} \left( \tau^2 - (k^2+\alpha^2)\right) <
\left(j_{0 , 1}+\tau\right)^2,
\qquad\text{for some }\tau>0.
\]
The conditions in \eqref{eq:_altre_coerc} follow by taking either
$\tau^2 = k^2+\alpha^2$, or $\tau \to +\infty$, respectively.
\end{proof}
\begin{corollary}
Let $\alpha,\mu,\omega$ be fixed, with
\begin{equation}\label{eq:mugei}
\mu < \left(j_{0 , 1}+1\right)^2.
\end{equation}
Then \eqref{eq:cond_res} holds true for every $k\neq0$.
\end{corollary}
We are ready to state and prove the main result of this section. For any $\Phi \in
\mathrm{Lip}([0,2\pi])$ we denote the Fourier coefficients of $e^{-\alpha x}\Phi(x)$ as
\[
\phi_k = \frac{1}{2\pi}\int_0^{2\pi} e^{-(ik+\alpha) x}\Phi(x)\,dx, \qquad k\in\Z.
\]
\begin{proposition} \label{prop:existence}
	Let $\alpha,\mu,\omega$ be fixed and $\Phi \in \mathrm{Lip}([0,2\pi])$. Let us assume that
	\begin{itemize}
		\item $\mu < \left(j_{0 , 1}+1\right)^2\simeq 3.4^2$,
		\item $\Phi(0)=\Phi(2\pi)=0$ and $\phi_0=\int_0^{2\pi} e^{-\alpha x}\Phi(x)\,dx=0$.
	\end{itemize}
	Then the functions
	\begin{equation} \label{eq.v}
	w(x,y)= \sum_{k\in \Z\setminus\{0\}}  \phi_k X_{k}(y) e^{ikx},\qquad
	v(x,y)= e^{\alpha x} w(x,y),
	\end{equation}
	where the functions $X_{k}$ are as in Lemmas \ref{lem:weak_coerc}, \ref{lem:bessel}, satisfy
	\begin{enumerate}
		\item $w \in H^1(\{(x,y) \in \R \times \R^+ : a < x + l y < b \})$ for any $l \in \R$ and $a<b$, and it solves \eqref{eq:linear_in_the_halfplane_w};
		\item $v \in H^1(\{(x,y) \in \R \times \R^+ : a < x + l y < b \})$ for any $l$ such that $l\alpha \geq 0$ and for every $a<b$, and it solves \eqref{eq:linear_in_the_halfplane};
		\item both $v$ and $w$ are analytic in $\R\times\R^+$, and $C^{0,\alpha}$ up to $y=0$, for every $\alpha<1$.
	\end{enumerate}
\end{proposition}
\begin{proof}[Proof of Proposition \ref{prop:existence}]
	In view of Lemma \ref{lem:weak_coerc}, we have that all the terms in the series in \eqref{eq.v}
	are smooth and satisfy the differential equations in \eqref{eq:linear_in_the_halfplane_w}.
	We now show that the  series converges in $H^1$, ensuring that
	$w$ satisfies the corresponding equation too. We start by observing that, by construction, the family $\{(x,y) \mapsto X_k(y) e^{ikx}\}_{k \in \Z \setminus\{0\}}$ is orthogonal in $H^1(S)$, $S=(0,2\pi) \times \R^+$, and in particular, for any $k, h \in \Z \setminus \{0\}$ and $k \neq h$, we have
	\[
	\int_{S} X_{k}(y) e^{ikx} \cdot  \overline{(X_{h}(y) e^{ihx})} = 0, \qquad \int_{S} X_{k}'(y) e^{ikx} \cdot  \overline{(X_{h}'(y) e^{ihx})} = 0
	\]
	and, recalling \eqref{eq:decad},
	\[
	\int_{S} \left|X_{k}(y) e^{ikx} \right|^2 \le \frac{C}{|k|},
	\qquad
	\int_{S} \left|X_{k}'(y) e^{ikx} \right|^2 \le C|k|,
	\qquad
	\int_{S} \left|X_{k}(y) \left(e^{ikx}\right)' \right|^2 \le C|k|.
	\]
	On the other hand, since $x \mapsto e^{-\alpha x} \Phi(x)$ can be extended to a $2\pi$-periodic Lipschitz continuous function, it is an $H^1$-function on $\mathbb S^1$, and
	its Fourier coefficients $\phi_k$ satisfy
	\[
	\sum_{k \in \Z} k^2|\phi_k|^2 < +\infty
	\]
	(recall that $\phi_0 = 0$). Combining the above inequalities we infer
	\[
	\left\| \sum_{k\neq0}   W_k(y)e^{ikx} \right\|_{H^1( S)}^2 \leq C \sum_{k \geq 1} (|\phi_k|^2 + |\phi_{-k}|^2) \left( \frac{1}{|k|} + |k|\right) < +\infty.
	\]
	We conclude that the series defining $w$ converges in $H^1(S)$, making $w$ a weak solution of \eqref{eq:linear_in_the_halfplane_w}. Since $w$ is periodic in the $x$ direction, we deduce
	that it belongs to $H^1((a,b)\times \R^+)$ for every $a<b$. Exploiting once again the periodicity in $x$ of $w$, we can readily infer that $w \in H^1(\{(x,y) \in \R \times \R^+ : a < x + l y < b \})$ for any $l\in \R$ and $a<b$. Moreover, by elliptic regularity,  $w$ is analytic in $\R\times \R^+$ and H\"older continuous up to the boundary. Analogous conclusions for the function $v$ can be drawn by the fact that $v(x,y)= e^{\alpha x} w(x,y)$, the only difference being that we need to exploit the assumption $l\alpha \geq 0$ in order to estimate the exponential factor.
\end{proof}
We conclude this section by showing that the Fourier expansions of the functions $w$ and $v$ can be exploited to give a description of their nodal sets for $y$ large.
\begin{lemma}\label{lem H1 and nodal}
	We consider again the assumptions of Proposition \ref{prop:existence}. Let $n \ge 1$ be the largest integer such that
	\[
		\phi_k = 0 \quad \forall |k| < n.
	\]
	Then there exists $y^* > 0$ and $2n$ disjoint simple curves $\Gamma_1, \dots, \Gamma_{2n}$ such that
	\begin{equation}\label{eqn gamma up}
		\{ (x,y) \in \R \times (y^*, +\infty) : w(x,y) = 0 (=v(x,y)) \} = \bigcup_{\substack{j = 1, \dots, 2n\\ h \in \Z}} \Gamma_j + (2\pi h, 0).
	\end{equation}
	The curves $\Gamma_j$ are asymptotic to evenly spaced parallel lines:
	there exist $\beta \in \R$  such that
	\[
	(x,y) \in \Gamma_j
	\qquad\iff\qquad
	\alpha y + n x = \beta + \pi j + o_y (1)\text{ as }y\to+\infty.
	\]
\end{lemma}
\begin{proof}
	By Lemma \ref{lem est xk} we have that
	\[
	\sup_{ (x,y)\in \R \times \R^+}|w(x,y)| \leq \sup_{y>0}\sum_{k\geq n}  |\phi_k| |X_{k}(y)|  +   |\phi_{-k}| |X_{-k}(y)| \leq C \sum_{k\geq n} (|\phi_k|   +   |\phi_{-k}|) < +\infty
	\]
	which implies that the series converges also uniformly in $\R \times \R^+$. Moreover we can extract the first term of the series and see that
	\[
	\left| w(x,y) - \phi_{n} X_{n}(y) e^{i n x}  - \phi_{- n} X_{-n}(y)e^{-i n x} \right| \leq C\sum_{k \geq n + 1}  (|\phi_k|   +   |\phi_{-k}|)e^{-ky} \leq C e^{-(n + 1 ) y}
	\]
	(see Corollary \ref{coro:decad}). This, in turn, implies that
	\begin{equation}\label{eq:zeriw}
	w(x,y) = \phi_{n} X_{n}(y) e^{i n x}  + \phi_{- n} X_{- n}(y)e^{-i n x} + O(e^{-(n + 1) y})
	\end{equation}
	uniformly in $x \in \R$. 
	
	We claim that the nodal lines of the functions $w$ (and of $v$) align asymptotically with those of the function
	\[
	\begin{split}
	(x,y) \mapsto A_{n}(x,y) &= \phi_{n} X_{n} (y) e^{i n x}  + \phi_{-n} X_{- n}(y)e^{-i n x}\\
	&= \phi_{n} C_n e^{-(n - i \alpha) y + i n x} + \phi_{-n} C_{-n} e^{(-n-i \alpha) y - i n x} + O(e^{-(n+2)y}) \\&=e^{-n y} \left( a_{n} \cos(\alpha y + nx) + b_{n} \sin(\alpha y + n x)  + O(e^{-2y})\right)\\
	&=e^{-ny} \left(\sqrt{a_n^2+b_n^2} \sin\left(\alpha y + nx - \beta \right) + O(e^{-2y})\right)
	\end{split}
	\]
	where the coefficients $a_n, b_n$ and $\beta$ are real numbers, $a_n^2+ b_n^2\neq0$ by assumption, and $\sin \beta = - a_n/\sqrt{a_n^2 + b_n^2}$. 
	Indeed, recalling \eqref{eq:zeriw} we have that, as $y\to+\infty$,
	\[
	e^{ny} w(x,y) =  \sqrt{a_n^2+b_n^2} \sin(\alpha y + n x -\beta) + O(e^{- y}).
	\]
	Analogously, one can show that also the series of the derivatives converges uniformly in $x\in \R$, and that, as $y\to+\infty$,
	\[
	e^{ny} w_x(x,y) =  n\sqrt{a_n^2+b_n^2} \cos(\alpha y + n x -\beta) + O(e^{- y}).
	\]
	By the implicit function theorem, there exists $y^* > 0$ sufficiently large such that the nodal set of the function $w$ in $\R \times (y^*, +\infty)$ is a countable union of graphs with respect to the $y$ variable, each one asymptotic to
	\[
	\alpha y + n x  = \beta + h \pi \qquad \text{for some $h \in \Z$}. 
	\]
	We choose $\Gamma_j$, $j=1,\dots,2n$, as $2n$ consecutive curves in this family of graphs, by taking $h=j$.
\end{proof}
\begin{remark}\label{rem:onlyif}
If the number of nodal zones for $y$ small is different from $2n$, then the nodal lines of $v$ 
must intersect. As a consequence, condition b) in Section \ref{sec:disk2halfplane} fails for 
such a $v$, which can not correspond to any element of $\Scal_\rot$ via Proposition \ref{prop:back2disk}. 
\end{remark}

\section{Nodal sets in the half-plane}\label{sec:nodalset}

In this section we study in detail the nodal structure of the function $v$ constructed in Proposition \ref{prop:existence}. For this purpose, we let
\[
    \mathcal{N} = \{(x,y) \in \R\times\R_+ : v(x,y) = 0\}
\]
be the nodal set of $v$, and we call a \emph{nodal component} of $v$ any connected component of 
$\R\times\R^+\setminus\Ncal$. 

We state the main result of this section. Its assumptions should be compared to those of Proposition \ref{prop:existence}, in particular we point out that they imply the existence of a unique solution $v$ of \eqref{eq:linear_in_the_halfplane}. We recall that for $\Phi \in
\mathrm{Lip}([0,2\pi])$ we denote the Fourier coefficients of $e^{-\alpha x}\Phi(x)$ as
\[
\phi_k = \frac{1}{2\pi}\int_0^{2\pi} e^{-(ik+\alpha) x}\Phi(x)\,dx, \qquad k\in\Z.
\]
\begin{proposition}\label{prop:nod}
Let $\alpha,\mu,\omega$ be fixed real numbers, $\Phi \in \mathrm{Lip}([0,2\pi])$ and $n \geq 1$ be a given integer. Let us assume that
\begin{itemize}
	\item the function $\Phi$ changes sign $2n$ times in $[0,2\pi]$, more precisely there exist $x_1 = 0 < x_2 < \dots < x_{2n+1} = 2\pi$ such that
	\[
	\begin{split}
		{\{x \in (0,2\pi) : \Phi(x) > 0\}} &= \bigcup_{k=0}^{n-1} (x_{2k+1}, x_{2k+2}) \quad \text{and} 
		\\ 
		{\{x \in (0,2\pi) : \Phi(x) < 0\}} &= \bigcup_{k=0}^{n-1} (x_{2k+2}, x_{2k+3});
	\end{split}
	\]
	\item the coefficients of the equation verify $\mu < \pi^2$;
	\item we have the following compatibility condition
	\begin{equation}\label{eq:compcond}
		\sup\{ |k| : \phi_k = 0\}  = n-1\geq 0.
	\end{equation}
\end{itemize}
Moreover, let $v$ denote the solution of \eqref{eq:linear_in_the_halfplane_w}, whose 
existence is guaranteed by Proposition \ref{prop:existence}.

Then there exist $2n$ connected, open sets 
$\cc_1, \dots, \cc_{2n} \subset \R \times \R^+$ such that:
\begin{itemize}
    \item extending the definition of $\cc_k$, by periodicity, 
    as $\cc_{k+2n} = \cc_{k} +(2\pi,0)$, $k\in\Z$, we have
    \[
    \begin{split}
    \cc_k\cap\cc_h&=\emptyset \qquad\text{for every }k\neq h, \text{ and}\\
	\overline{\cc}_k\cap\overline{\cc}_h&\neq\emptyset \qquad\iff\qquad k-h=-1,0,1;
    \end{split}
	\]
    \item any nodal component of $v$ is one of the $\cc_k$:
    \[
        \R\times\R^+\setminus \mathcal{N} = \bigcup_{k\in\Z} \cc_k;
    \]
    \item each of them  touches the $x$-axis in a single (connected) interval: 
    \[
       \overline{\cc}_k \cap \{(x,0)\} = [x_k,x_{k+1}] \qquad \text{for any $k = 1, \dots 2n$}; 
    \]
    \item they are asymptotic to a family of evenly spaced strips: there exists $\beta \in \R$
    such that
    \[
	\cc_{k} \subset \{(x,y) : \beta + \pi k + o_y (1) < \alpha y + n x < \beta + \pi (k+1) + o_y (1)\}
	\text{ as }y\to+\infty.
	\]
\end{itemize}
\end{proposition}

The remaining part of this section is devoted to the proof of Proposition \ref{prop:nod}. We shall prove it in a series of intermediate steps. First we briefly investigate the local structure of the nodal set $\mathcal{N}$.
\begin{lemma}\label{lem:analiticity}
Under the above notation:
\begin{itemize}
\item $\mathcal{C} = \{(x,y) \in \R\times\R_+  : v(x,y) = 0, \nabla v (x,y) = 0\}
$ is discrete in $\R\times\R^+$;
\item $\mathcal{N} \setminus \mathcal{C}$ is the union of countably many analytic curves;
\item if $\Phi(\bar x)\neq 0$ and $l \in \R$, then the set
\[
\mathcal{N} \cap \{ (x,y)  : x +ly = \bar x \}
\]
is discrete, and it does not accumulate at $\{y=0\}$.
\end{itemize}
\end{lemma}
We point out that, for the moment, it may still be that $\mathcal{C}$ accumulates at some point of the discrete set $\{(x,0):\Phi(x)=0\}$. 
\begin{proof}
We recall that $v$ satisfies \eqref{eq:linear_in_the_halfplane}, $v$ is
analytic in $\R\times\R^+$ and continuous up to the boundary $\{(x,y):y=0\}$ (see Proposition \ref{prop:existence}). By well known results of Hartman and Wintner \cite{MR58082}, 
the set $\mathcal{C}$ is discrete in 
$\R\times\R^+ $. 

As a consequence, by the analytic implicit function theorem,
$\mathcal{N} \setminus \mathcal{C}$ is the disjoint union of countably many analytic
curves which are either unbounded, accumulate at some point of $\{(x,0):\Phi(x)=0\}$,
or meet each other at points of $\mathcal{C}$. 

Finally, let $\varphi:[0,+\infty)\to\R$ be
defined as 
\[
\varphi(y) = v(\bar x - ly,y).
\] 
Then $\varphi$ is real analytic for $y>0$, continuous up to $y=0$ and $\varphi(0)\neq0$. We deduce that its zero set is discrete. Since
\[
\mathcal{N} \cap \{ (x,y)  : x +ly = \bar x \} \equiv 
\{ (\bar x - ly,y)  : \varphi(y) = 0 \}.
\]
the lemma follows.
\end{proof}
Let $A$ be any nodal component of $v$. In the following, for any $h\in\Z$, we write
\[
A_h = A - (2h\pi,0).
\]
Since $v$ is $2\pi$-periodic in $x$, $A_h$ is itself a nodal component of $v$. As a consequence,
either $A$ and $A_h$ coincide, or they are disjoint. We prove that this property is independent 
of $h\neq0$.
\begin{lemma}\label{lem:moltocorto}
Let $A$ be any nodal component of $v$. Then
\begin{itemize}
\item either $A\equiv A_h$ for some $h\in\Z$, in which case $A\equiv A_k$ for every $k\in\Z$,
\item or $A\cap A_h=\emptyset$ for some $h\in\Z$, in which case $A\cap A_k=\emptyset$ for every $k\neq0$, and 
\[
\sup_{y>0} |\{ x : (x,y) \in A\}| \leq 2\pi.
\]
\end{itemize}
\end{lemma}
\begin{proof}
We start by examining the first alternative. Let $(\bar x, \bar y) \in A\equiv A_h$, 
with $h\ge1$, so that also 
$(\bar x + 2h\pi, \bar y)\in A$. By connectedness, there exists a curve $\gamma\subset A$  
joining $(\bar x, \bar y)$ and $(\bar x+2h\pi, \bar y)$. Since $2h\pi/2\pi=h\in\N$, by the universal chord theorem (see e.g.\ \cite{MR299735}), there exists $(x_1,y_1),(x_2,y_2)\in\gamma$ such that $(x_2,y_2)=(x_1,y_1)+(2\pi,0)$. Thus $A\cap A_{1}\ni(x_2,y_2)$, which implies $A\equiv A_k$ for every $k\in\Z$.

Conversely, let us assume that $A\cap A_k=\emptyset$ for every $k\neq0$. Then, for every $y>0$, 
\[
\{ x : (x,y) \in A\} = \bigcup_{k\in\Z} \{ x\in[2k\pi,2(k+1)\pi) : (x,y) \in A\}= 
\bigcup_{k\in\Z} \{ x\in[0,2\pi) : (x,y) \in A_k\},
\]
and such union is disjoint by assumption. We deduce that $|\{ x : (x,y) \in A\}| \leq 
|[0,2\pi)|$.
\end{proof}
To proceed, we need the following result, which is a consequence of a Poincar\'e-type inequality
(see Lemma \ref{lem eigen rect}).
\begin{lemma}\label{lem cc and eigen}
	Let $A$ be any nodal component of $v$ and assume that $\left.v\right|_{A} \in H^1_0 (A)$ and 
	\[
		\sup_{y>0} |\{ x : (x,y) \in A\}| \leq 2 \pi.
	\]
	Then necessarily $\mu \geq \pi^2$.
\end{lemma}
\begin{proof}
	By assumption the function $v \in H^1(A)$ verifies
	\[
	\begin{cases}
	-\Delta v + \omega e^{-2y} v_x = e^{-2y}\mu v & \text{in $A$}\\
	v = 0              & \text{on $\partial A$}.
	\end{cases}
	\]
	Multiplying by $v$ and integrating by parts over $A$ yields the identity
	\[
	\int_{A} |\nabla v|^ 2 = \mu \int_A e^{-2y}v^2,
	\]
	indeed 
	\[
	\frac{\omega}{2}\int_A e^{-2y}(v^2)_x = 0
	\]
	for every $v\in H^1_0(A)$, by density of the test functions. 
	
	We argue by Steiner symmetrization with respect to the $y$-axis, see e.g.\
	\cite{Kawohl}. We stress that the weight $(x,y) \mapsto e^{-2y}$ is independent of the $x$ variable. Let $A^* \subset (-\pi,\pi) \times \R^+$ be defined as
	\[
	A^* := \left\{(x,y): y>0,\ |x| < |\{ x : (x,y) \in A\}|/2 \right\}.
	\]
	and  $v^* \in H^1_0(A^*) \subset H^1_0((-\pi,\pi)\times \R^+)$ be the Steiner symmetrization of the function $v|_A$. By well-known properties of the Steiner symmetrization, see, we obtain
	\[
	\int_{(-\pi,\pi)\times \R^+} |\nabla v^*|^ 2 \leq \mu \int_{(-\pi,\pi)\times \R^+}
	e^{-2y}(v^*)^2.
	\]
	Since $v$ and $v^*$ are not identically zero, by Lemma \ref{lem eigen rect} we obtain
	\[
	\mu \ge  \left(j_{\frac{1}{2},1}\right)^2 =  \pi^2.\qedhere
	\]
\end{proof}
\begin{lemma}\label{lem:infty<2pi}
Let $y^*$ be defined as in Lemma \ref{lem H1 and nodal} and let $A$ denote any nodal 
component of $v$ such that $A \cap \{(x,y):y>y^*\}\neq \emptyset$. Then
\[
\sup_{y>0} |\{ x : (x,y) \in A\}| \le 2 \pi.
\]
\end{lemma}
\begin{proof}
Without loss of generality we can assume that $v>0$ in $A$ and, by Lemma \ref{lem H1 and nodal}, 
there exists a half-line $\ell:=\{(x,y):y\ge y^*,\ \alpha y + n x = q \}$ such that 
$\ell \subset A$.
Let us assume by contradiction that $\sup_{y>0} |\{ x :
(x,y) \in A\}| > 2 \pi$. By Lemma \ref{lem:moltocorto} we deduce that $A$ is $2\pi$-periodic in the $x$-direction, so that also $\ell+(2\pi,0)\subset A$. By connectedness, we
can find a simple curve $\gamma$ such that
\[
\gamma\subset A,\quad \gamma\cap\{(x,y):y\ge y^*\} = \ell\cup\ell+(2\pi,0)
\quad \text{and} \quad
\gamma\cap\{(x,y):y\le y^*\} \text{ is compact.}
\]
As a consequence,
$\R\times\R^+\setminus\gamma = O_0\cup O_1$,  where each $O_i$ is open and connected and only one of them, say $O_1$, is such that 
\[
O_1 \supset \{(x,y^*) : x^* < x < x^*+2\pi \}\neq \emptyset,\qquad\text{where }\alpha y^* + 
n x^* = q.
\]
Since $\gamma \cap \{y\le y^*\}$ is compact, we deduce that there exist
$q_1,q_2$ and $y_0>0$ such that
\begin{equation}\label{eq:O1}
O_1  \subset \{(x,y):y\ge y_0,\ q_1 < \alpha y + n x < q_2 \}.
\end{equation}
Now, let $B\neq A$ be any other nodal component of $v$ satisfying
$B\subset O_1$ ($B$ exists as $v$ changes sign in $O_1$, by Lemma \ref{lem H1 and nodal}). Then $B$ can not be periodic in the
$x$-direction and hence, by Lemma \ref{lem:moltocorto}, $\sup_{y>0} |\{ x : (x,y) \in B\}|
\le 2 \pi$. By Proposition \ref{prop:existence} and  \eqref{eq:O1} we have that $\left.v\right|_{B}
\in H^1_0(B)$. Thus Lemma \ref{lem cc and eigen} applies, providing a contradiction since we are assuming $\mu<\pi^2$.
\end{proof}
In the same spirit, we show the following.
\begin{lemma}\label{lem:pochezoneinfinite}
Let $y^*$ be defined as in Lemma \ref{lem H1 and nodal} and let $A$ denote any nodal 
component of $v$ such that $A \cap \{(x,y):y>y^*\}\neq \emptyset$. Then
$A \cap \{(x,y):y>y^*\}$ is connected.
\end{lemma}
\begin{proof}
The proof follows the lines of that of Lemma \ref{lem:infty<2pi}.
Assume by contradiction that $A \cap \{(x,y):y>y^*\}$ contains at least two connected components, say $A_1$ and $A_2$. Then, by Lemma \ref{lem H1 and nodal} we can find half-lines $\ell_j:=\{(x,y):y\ge y^*, \alpha y + n x = q_j \} \subset A_j$, and a simple curve $\gamma\subset A$ which joins such
half lines. Then  $\R\times\R^+\setminus\gamma$ is the disjoint union of $O_0$ and $O_1$, and one can find a contradiction as above.
\end{proof}

Motivated by Lemma \ref{lem:pochezoneinfinite} we introduce the following notation.
\begin{definition}\label{def:cc}
Let $y^*>0$ and $\beta\in\R$ be fixed as in Lemma \ref{lem H1 and nodal}. We denote with 
$\cc_{k}$, $k\in\Z$, the nodal component of $v$ asymptotic to 
\[
\{(x,y) : \beta + \pi k  < \alpha y + n x < \beta + \pi (k+1) \}
\text{ as }y\to+\infty.
\]
\end{definition}

By Lemma \ref{lem:pochezoneinfinite} we have that $\cc_{k}$ and $\cc_h$ are disjoint, as long as 
$h\neq k$. To conclude the proof of Proposition \ref{prop:nod} we are left to show that the sets 
$\cc_k$ exhaust the nodal components of $v$. At the moment we can not assure that each $\cc_k$ intersects 
the $x$-axis. However, in such case, the horizontal order is preserved.
\begin{lemma}\label{lem:ordinate}
Let $\cc_{k_1}$, $\cc_{k_2}$ be two nodal components of $v$ as in Definition \ref{def:cc}, and let $k_1<k_2$. If $ \overline{\cc}_{k_i}\cap\{(x,0)\}\neq\emptyset$, $i=1,2$, 
then 
\[
(\widehat x_i,0) \in \overline{\cc}_{k_i}
\quad\implies\quad
\widehat x_1 < \widehat x_2.
\]
\end{lemma}
\begin{proof}
This follows by connectedness since the segments $\cc_{k}\cap\{(x,y^*)\}$ are ordered according to the index $k$.
\end{proof}
\begin{lemma}\label{lem:infty in strip}
Let $A$ denote any nodal component of
$v$. There exist
$q_-<q_+$ such that
\[
A \subset \{(x,y): q_- < \alpha y + n x < q_+ \}.
\]
\end{lemma}
\begin{proof}
We only show that $A \subset \{(x,y):  \alpha y + n x < q_+ \}$, for some $q^+$,
because the other property follows by a similar argument. In the following, we fix $x_0$ such that $\Phi(x_0)\neq0$, and we write
\[
\ell :=\{(x,y):y > 0,\ \alpha y + n (x-x_0) = 0 \},\qquad
L^- := \{(x,y):y > 0,\ \alpha y + n (x-x_0) < 0 \}.
\]
Moreover, by Lemma \ref{lem H1 and nodal}, we can assume that $v$ does not vanish 
on $\ell\cap\{(x,y):y\ge y^*\}$.  

We have to show that, for some $h\in \Z$,
\[
A_h := A - (2h\pi , 0) \subset L^-.
\]
To start with, we observe that $A_h \cap L^- \neq \emptyset$ for every $h \ge \bar h$ sufficiently large (indeed $A$ is not empty). Let us assume by contradiction that
$A_h \setminus L^- \neq \emptyset$ for every $h \ge \bar h$ as well. By connectedness,
we obtain that $I_h:= \ell \cap A_h$ is non-empty, relatively open in $\ell$, and
with  non-empty (relative) boundary $\partial I_h\subset\Ncal$. Finally,
by Lemmas \ref{lem:infty<2pi} and \ref{lem:moltocorto}, we have that $I_{h_1} \cap  I_{h_2}
=\emptyset$ for every $h_1\neq h_2$. We deduce that the set
\[
\bigcup_{h\ge \bar h} \partial  I_h \subset  \left(\mathcal{N}\cap\ell\cap\{y\le y^*\}\right) \quad\text{ is infinite.}
\]
This contradicts the last part of Lemma \ref{lem:analiticity}.
\end{proof}
\begin{lemma}\label{lem:vH1davvero}
Let $A$ denote any nodal component of $v$. Then
$\left.v\right|_{A} \in H^1(A)$.
\end{lemma}
\begin{proof}
This follows by Lemma \ref{lem:infty in strip} and Proposition \ref{prop:existence}.
\end{proof}
\begin{lemma}\label{lem:inftyconnessecon0}
Let $\cc_k$ be a nodal component of $v$ as in Definition \ref{def:cc}. 
Then $\left.v\right|_{\partial \cc_k}\not\equiv0$. In particular, 
\[
\{ x : (x,0) \in \overline{\cc}_k\}
\qquad\text{contains a non-trivial interval.}
\]
\end{lemma}
\begin{proof}
The lemma follows by Lemmas \ref{lem cc and eigen}, \ref{lem:infty<2pi} and \ref{lem:vH1davvero}.
\end{proof}
\begin{lemma}\label{lem:tuttemenodipi}
Let $A$ denote any nodal component of $v$. Then
\[
\sup_{y>0} |\{ x : (x,y) \in A\}| \le 2 \pi.
\]
\end{lemma}
\begin{proof}
Let $A$ contradict the result; then $A\equiv A+(2\pi,0)$ (Lemma \ref{lem:moltocorto}) and $A \subset 
\{(x,y):y<y^*\}$ (Lemma \ref{lem:infty<2pi}). As a consequence, there exists a simple curve 
$\gamma\subset A$, with $\gamma+(2\pi,0)\equiv \gamma$. Then  $\R\times\R^+\setminus\gamma = O_0\cup O_1$,  where each $O_i$ is open and connected and $O_1\supset\{(x,y):y\ge y^*\}$. Now, let $A'$ be any 
nodal region of $v$ intersecting $\{(x,y):y\ge y^*\}$. Then $A\cap A'=\emptyset$. By Lemma \ref{lem:inftyconnessecon0} there exists
$\gamma'\subset A'$ with one endpoint in $O_1$ and the other one in $O_0$, so that $\gamma'$ intersects 
$\gamma$, a contradiction.
\end{proof}
\begin{lemma}\label{lem:tutteconnessecon0}
Let $A$ denote any nodal component of $v$. 
Then $\left.v\right|_{\partial A}\not\equiv0$. In particular, 
\[
\{ x : (x,0) \in \overline{A}\}
\qquad\text{contains a non-trivial interval.}
\]
\end{lemma}
\begin{proof}
The lemma follows by Lemmas \ref{lem cc and eigen}, \ref{lem:tuttemenodipi} 
and \ref{lem:vH1davvero}.
\end{proof}
We are ready to conclude the proof of the main result of the section.
\begin{proof}[End of the proof of Proposition \ref{prop:nod}]
	We are left to show that the sets $\cc_k$ (Definition \ref{def:cc}) exhaust the nodal 
	components of the function $v$ so that, in particular, for each $\cc_k$ there exists 
	two consecutive zeros of the function $\Phi$, $x_j < x_{j+1} \in [0,2\pi]$, and $h \in \Z$ such that
	\[
		\overline{\cc}_k \cap \{(x,0)\} = [x_j, x_{j+1}] + (2h\pi,0).
	\]
	Let $\cc_k$ be any connected component as in Definition \ref{def:cc}, then by Lemma 
	\ref{lem:tutteconnessecon0} and continuity of the function $v$ (see Proposition \ref{prop:existence}) 
	there exist two consecutive zeros $x_j < x_{j+1}$ and $h \in \Z$ such that 
	\[
		[x_j, x_{j+1}] + (2h\pi,0) \subset \overline{\cc}_k \cap \{(x,0)\}.
	\]
	By periodicity in the $x$-direction, it follows that
	\[
	[x_j, x_{j+1}] + (2(h+1)\pi,0) \subset \overline{\cc}_{k+2n} \cap \{(x,0)\}.
	\]
	Now, on the one hand, for $y \geq y^*$ we already know that the nodal set of $v$ between 
	$\cc_k$ (included) and 
	$\cc_{k+2n}$ (excluded) is precisely given by the $2n$ sets $\cc_{k}, \dots, \cc_{k+2n-1}$.
	On the other hand, for $y =0$ the nodal set of $v$ between $(x_j + 2h\pi,0)$ and 
	$(x_j + 2(h+1)\pi,0)$ consists in exactly $2n$ intervals. Once again, we appeal to Lemma 
	\ref{lem:inftyconnessecon0} to infer that every $\cc_{k}, \dots, \cc_{k+2n-1}$ contains exactly 
	one interval on $\{(x,0)\}$, and the intersections are ordered by Lemma \ref{lem:ordinate}. 
	The remaining conclusions follow straightforwardly.
\end{proof}

\section{End of the proof of Theorem \ref{thm:main_intro}}\label{sec:proofmain}

We give the proof in the case $K=2n$ is even. The odd case can be treated with minor 
changes, see the discussion at the beginning of Section \ref{sec:sepvar}. 

In view of Proposition \ref{prop:back2disk}, the existence of an element of $\Scal_\rot$, as
defined in \eqref{eq:S_om}, with the required nodal properties is equivalent to the existence of 
a solution of \eqref{eq:linear_in_the_halfplane}, having trace
\begin{equation}\label{eq:traceproof}
\Phi(x) =  \sum_{m=1}^K  \frac{(-1)^{m+1}}{l_m} s_m \varphi_m
\end{equation}
(recall equations \eqref{eq:def_li}, \eqref{eq:def_v}), and enjoying properties b) and c) in 
Section \ref{sec:disk2halfplane} (property a) is already contained in 
\eqref{eq:linear_in_the_halfplane}).

The existence of such functions is provided by Proposition \ref{prop:existence}, while properties 
b) and c) follow from Proposition \ref{prop:nod}, once $\Phi$ satisfies the compatibility 
conditions \eqref{eq:compcond}, i.e.\  
\begin{equation}\label{eq:main_cond_in_proof}
\phi_k = \frac{1}{2\pi}\int_0^{2\pi} e^{-(ik+\alpha) x}\Phi(x)\,dx=0, \quad |k|<n,
\qquad\text{and $\phi_n\neq0$ }
\end{equation}
(or equivalently  $\phi_{-n}=\overline{\phi_n}\neq0$).
Under the validity of these conditions, also the asymptotic 
expansion \eqref{eq:expans_intro} follows from Proposition \ref{prop:nod} and the definition 
of the map $\Tcal$ (equation \eqref{eq:defT}), see also Remark \ref{rem:back2traces}. The details 
of these calculations are very similar to those in \cite[Proof of Thm. 1.5]{MR4020313}

Writing $c_m = s_m/l_m$ in \eqref{eq:traceproof} and \eqref{eq:main_cond_in_proof}, and 
recalling also Remark \ref{rem:onlyif}, we obtain 
that Theorem \ref{thm:main_intro} is equivalent to the following assertion: \emph{there exists $\bar c=(\bar c_1, \dots, \bar c_{2n})$, with $(-1)^{m+1} c_m>0$, such that 
\[
	\sum_{m  = 1}^{2n}  \frac{1}{2\pi} \int_0^{2\pi} e^{-(ik+\alpha)x} c_m \bound_m(x) dx = 0, \quad |k|<n,
\]
and 
\[
	\sum_{m  = 1}^{2n}  \frac{1}{2\pi} \int_0^{2\pi} e^{-(in+\alpha)x} c_m \bound_m(x) dx \neq 0
\]
if  and only if $c=t\bar c$.}

To prove this last claim, let us define the matrix $A \in \C^{2n \times 2n}$
\[
\begin{split}
	A &= (a_{km})_{\substack{k=-n+1, \dots, n\\ m = 1, \dots, 2n} } = \left(
		\frac{1}{2\pi} \int_0^{2\pi} e^{-(ik+\alpha)x} \bound_m(x) dx \right)_{km} \\
	  &= \left(
	  \frac{1}{2\pi} \int_0^{2\pi} e^{-(ik+\alpha)\var_m} \bound_m(\var_m) d\var_m \right)_{km}.
\end{split}	
\]
Observe that we have suitably renamed the dummy variables in each integral as, later, this will lead us to more manageable identities. We can write the set of compatibility conditions \eqref{eq:main_cond_in_proof} as a system of linear equations,
\begin{equation}\label{eqn lin sys}
A \begin{pmatrix}
c_1 \\ c_2 \\ \dots \\ c_{2n}
\end{pmatrix} = \begin{pmatrix} 0 \\ 0 \\ \dots \\ \phi_n
\end{pmatrix}.
\end{equation}
To show our claim, we prove that the matrix $A$ is invertible, and that it is possible to choose $\phi_n \neq 0$ such that the solution vector is real and sign-alternating. First, exploiting the multi-linearity of the determinant, we have
\[
	\det A =  \frac{1}{(2\pi)^{2n}} \int_{[0,2\pi]^{2n}} \prod_{m=1}^{2n} e^{-\alpha \var_m}\bound_m(\var_m) \cdot  \det A'
\]
where we have introduced the matrix
\[
	A' = \begin{pmatrix}
e^{-i(-n+1)\var_1} & e^{-i(-n+1)\var_2} & \dots & e^{-i(-n+1)\var_{2n}} \\
e^{-i(-n+2)\var_1} & e^{-i(-n+2)\var_2} & \dots & e^{-i(-n+2)\var_{2n}} \\
&&\dots&& \\
e^{-in\var_1} & e^{-in\var_2} & \dots & e^{-in\var_{2n}}
\end{pmatrix}.
\]
Factoring out the coefficients of the first row, we recognize Vandermonde's determinant and compute
\begingroup
\allowdisplaybreaks
\begin{align*}
\det A' &=  e^{-i(-n+1) \sum_{m=1}^{2n} \var_m} \left|\begin{matrix}
1 & \dots & 1 \\
e^{-i\var_1} & \dots & e^{-i \var_{2n}} \\
&\dots&& \\
e^{-(2n-1) i \var_1} & \dots & e^{-(2n-1) i \var_{2n}}
\end{matrix} \right|\\
&=  e^{-i(-n+1) \sum_{m=1}^{2n} \var_m} \prod_{1 \leq p < q \leq 2n} \left(e^{-i \var_q} - e^{-i \var_p}\right)\\
&=  e^{i(n-1) \sum_{m=1}^{2n}  \var_m} \prod_{1 \leq p < q \leq 2n} (-1) e^{-\frac12 i \var_q - \frac12 i \var_p} \left(- e^{-\frac12 i \var_q + \frac12 i \var_p} + e^{-\frac12 i \var_p + \frac12 i \var_q}\right) 
\\
&= e^{i(n-1) \sum_{m=1}^{2n}  \var_m} (-1)^{\frac{2n(2n-1)}{2}} e^{-\frac12 i(2n-1) \sum_{m=1}^{2n}  \var_m} \prod_{1 \leq p < q \leq 2n} \left( e^{-\frac12 i \var_p + \frac12 i \var_q} - e^{-\frac12 i \var_q + \frac12 i \var_p}\right) \\
&=  (-1)^{n}  (2i)^{\frac{2n(2n-1)}{2}} e^{-\frac12 i \sum_{m=1}^{2n}  \var_m} \prod_{1 \leq p < q \leq 2n} \left(\frac{e^{\frac12 i (\var_q -\var_p)} - e^{-\frac12 i (\var_q-\var_p)}}{2i}\right) \\
&=  (-1)^{n}  (2i)^{n(2n-1)}  e^{-\frac12 i \sum_{m=1}^{2n}  \var_m} \cdot \prod_{1 \leq p < q \leq 2n} \sin \left(\frac{\var_q -\var_p}{2}\right).
\end{align*}
\endgroup
Thus we find
\[
\det A = \frac{(-1)^{n}  (2i)^{n(2n-1)} }{(2\pi)^{2n}} \int_{[0,2\pi]^{2n}}  \underbrace{\prod_{m=1}^{2n} e^{-\alpha \var_m}\bound_m(\var_m) \prod_{1 \leq p < q \leq 2n} \sin \left(\frac{\var_q -\var_p}{2}\right)}_{\mathrm{Mod}} \cdot \underbrace{e^{-\frac12 i \sum_{m=1}^{2n}  \var_m}}_{\mathrm{Phase}} .
\]
We show that the integral in the previous expression is always different from 0. We recall that, by assumption, the functions $\bound_m$ are supported on ordered intervals. More precisely, using the notation introduced in Proposition \ref{prop:nod}, we have
\[
	\{ t \in [0, 2\pi] : \bound_m(t) > 0 \} = (x_m, x_{m+1}).
\]
As a result, the integral can be restricted to the open and not empty set
\[
	\mathcal{O} = (x_1, x_{2}) \times (x_2,x_3) \times \dots \times (x_{2n}, x_{2n+1}) \subset [0,2\pi]^{2n}.
\] 
Moreover, for any choice $1 \leq p < q \leq 2n$, in $\mathcal{O}$ we have $0 < \var_q-\var_p < 2\pi$ and thus
\[
	0 < \frac{\var_q-\var_p}{2} < \pi \implies \sin \left(\frac{\var_q -\var_p}{2}\right) > 0.
\]
As it turns out, the factor denoted as $\mathrm{Mod}$ is strictly positive in $\mathcal{O}$. This function corresponds to the modulus of the integral function. On the other hand, the factor $\mathrm{Phase}$ is complex and of modulus 1. Let us investigate more closely the argument of $\mathrm{Phase}$. We find
\[
	\sum_{m=1}^{2n} x_m < \sum_{m=1}^{2n} t_m < \sum_{m=1}^{2n} x_{m+1} = \sum_{m=1}^{2n} x_m + (x_{2n+1}-x_1) < \sum_{m=1}^{2n} x_m + 2\pi.
\]
that is, letting $X = \sum_{m=1}^{2n} x_m$, for any  $(t_1, \dots, t_{2n}) \in \mathcal O$
\[
	0 < \frac12 \left(\sum_{m=1}^{2n} t_m - X\right) < \pi.
\]
We can rewrite the determinant as
\[
\det A = C \left[ \int_{\mathcal{O}}  \mathrm{Mod} \cdot \cos {\frac12 \left( \sum_{m=1}^{2n}  \var_m - X\right)} -i \int_{\mathcal{O}}  \mathrm{Mod} \cdot \sin {\frac12 \left( \sum_{m=1}^{2n}  \var_m - X\right)} \right]
\]
for some complex constant $C \in \C \setminus \{0\}$. By the previous discussion, the second integral is positive. It follows that the determinant of $A$ is not zero, proving that the linear system \eqref{eqn lin sys} has a unique solution for any $\phi_n$.

We now show that there exists $\phi_n \neq 0$ such that the solution vector is real and sign-alternating. By Cramer's rule we have
\[
c_l = (\det A)^{-1} \det A_l
\]
where $A_l$ is the matrix obtained by replacing the $l$ column of $A$ with the right hand side of system \eqref{eqn lin sys}. Now, by the same considerations as before, we have
\[
\det A_l =  \frac{1}{(2\pi)^{2n}} \int_{[0,2\pi]^{2n}} \prod_{m=1, m \neq l}^{2n} e^{-\alpha \var_m}\bound_m(\var_m) \cdot  \det A'_l
\]
where 
\[
A'_l = \begin{pmatrix}
e^{-i(-n+1)\var_1} & e^{-i(-n+1)\var_2} & \dots & e^{-i(-n+1)\var_{l-1}} & 0 & e^{-i(-n+1)\var_{l+1}} & \dots & e^{-i(-n+1)\var_{2n}} \\
e^{-i(-n+2)\var_1} & e^{-i(-n+2)\var_2} & \dots & e^{-i(-n+2)\var_{l-1}} & 0 & e^{-i(-n+2)\var_{l+1}} & \dots & e^{-i(-n+2)\var_{2n}} \\
&&\dots&& \\
e^{-in\var_1} & e^{-in\var_2} & \dots & e^{-in\var_{l-1}} & \phi_n & e^{-in \var_{l+1}} & \dots & e^{-in\var_{2n}}
\end{pmatrix}
\]
Developing the determinant with respect to the $l$-th column, factoring out the first line and exploiting once more Vandermonde's determinant, we find
\begingroup
\allowdisplaybreaks
\begin{align*}
\det A_l' &=  (-1)^{l-1} \phi_n e^{-i(-n+1) \sum_{m=1, m\neq l}^{2n} \var_m} \left|\begin{matrix}
1 & \dots & 1 \\
e^{-i\var_1} & \dots & e^{-i \var_{2n}} \\
&\dots&& \\
e^{-(2n-2) i \var_1} & \dots & e^{-(2n-2) i \var_{2n}}
\end{matrix} \right|\\
&=  (-1)^{l-1} \phi_n e^{-i(-n+1) \sum_{m=1, m\neq l}^{2n} \var_m} \prod_{1 \leq p < q \leq 2n, p,q \neq l} \left(e^{-i \var_q} - e^{-i \var_p}\right)\\
&=  (-1)^{l-1} \phi_n  e^{i(n-1) \sum_{m=1, m\neq l}^{2n}  \var_m} \prod_{1 \leq p < q \leq 2n, p,q \neq l} (-1) e^{-\frac12 i \var_q - \frac12 i \var_p} \left(- e^{-\frac12 i \var_q + \frac12 i \var_p} + e^{-\frac12 i \var_p + \frac12 i \var_q}\right) 
\\
&=  (-1)^{l-1} \phi_n  e^{i(n-1) \sum_{m=1, m\neq l}^{2n}  \var_m} (-1)^{\frac{(2n-1)(2n-2)}{2}} e^{-\frac12 i(2n-2) \sum_{m=1, m \neq l}^{2n}  \var_m}\cdot
\\
&\phantom{aaaaaaaaaaaaaaaaaaaaaaaaaaaaaaaaaaaaaaaaa}
\cdot \prod_{1 \leq p < q \leq 2n, p,q \neq l} \left( e^{-\frac12 i \var_p + \frac12 i \var_q} - e^{-\frac12 i \var_q + \frac12 i \var_p}\right) \\
&=  (-1)^{l+n-2} \phi_n  (2i)^{\frac{(2n-1)(2n-2)}{2}} \prod_{1 \leq p < q \leq 2n, p} \left(\frac{e^{\frac12 i (\var_q -\var_p)} - e^{-\frac12 i (\var_q-\var_p)}}{2i}\right) \\
&=  (-1)^{l+n}  (2i)^{(2n-1)(n-1)}  \phi_n \cdot \prod_{1 \leq p < q \leq 2n, p,q \neq l} \sin \left(\frac{\var_q -\var_p}{2}\right)
\end{align*}
\endgroup
We obtain
\[
\begin{split}
c_l &= \frac{(\det A)^{-1} (-1)^{l+n}  (2i)^{(2n-1)(n-1)}  \phi_n }{(2\pi)^{2n-1}} \int   \prod_{m=1, m \neq l}^{2n} e^{-\alpha \var_m} \bound_m(\var_m) \prod_{1 \leq p < q \leq 2n, p, q \neq l} \sin \left(\frac{\var_q -\var_p}{2}\right)\\
&= (-1)^{l+1} \Gamma \int_{[0,2\pi]^{2n-1}}   \prod_{m=1, m \neq l}^{2n} e^{-\alpha \var_m} \bound_m(\var_m) \cdot \prod_{1 \leq p < q \leq 2n, p, q \neq l} \sin \left(\frac{\var_q -\var_p}{2}\right)
\end{split}
\]
where $\Gamma \in \C$. Reasoning as before, we see that the integral is always strictly positive. Thus $c_l$ satisfies the condition $(-1)^{l+1}c_l > 0$ if and only if $\Gamma$ is real and positive, $\Gamma = t > 0$. We obtain the solution
\[
c_l = t (-1)^{l+1} \int_{[0,2\pi]^{2n-1}}   \prod_{m=1, m \neq l}^{2n} e^{-\alpha \var_m} \bound_m(\var_m) \prod_{1 \leq p < q \leq 2n, p, q \neq l} \sin \left(\frac{\var_q -\var_p}{2}\right).
\]
and
\[
	\phi_n = t (-1)^{n+1} \frac{2^{2n-2} }{\pi} \int_{[0,2\pi]^{2n}}  \prod_{m=1}^{2n} e^{-\alpha \var_m}\bound_m(\var_m) \prod_{1 \leq p < q \leq 2n} \sin \left(\frac{\var_q -\var_p}{2}\right) \cdot e^{-\frac12 i \sum_{m=1}^{2n}  \var_m}.
\]

\begin{proof}[Proof of Corollary \ref{coro:intro}]
This follows by uniqueness of $\bar s$; indeed, notice that a rotation of $2\pi/K$ leaves the 
data unchanged, while the indexes of the densities are shifted by $1$. By uniqueness, 
$\bar s_{m} = \bar s_{m-1}$, for every $m$.
\end{proof}

\section{Single-mode special solutions}\label{sec:special}

In the following we deal with the fundamental single-mode solutions that we constructed by 
separation of variables in Section \ref{sec:sepvar}. Theorems \ref{prop:intr_res}, 
\ref{prop:intr_entire} will follow once again by Proposition \ref{prop:back2disk}.  

\subsection{The homogeneous Dirichlet problem}	

We now turn our attention to the homogeneous version of \eqref{eq:linear_in_the_halfplane}, that is we look for conditions under which there exists a non zero solution $v$ of
\begin{equation}\label{eq:linear_in_the_halfplane_HD}
	\begin{cases}
		-\Delta v + \omega e^{-2y} v_x = e^{-2y}\mu v & x\in\R,\,y>0\\
		v(x + 2\pi, y) = e^{2\pi \alpha}v(x,y)              & x\in\R,\,y\ge 0\\
		v(x,0)=0 & 0\le x \le 2\pi
	\end{cases}
\end{equation}
with nodal set consisting of $2k$ strips (up to horizontal $2\pi$-periodicity), $k\ge1$, that connect the boundary $y=0$ with $y \to +\infty$, as in the previous section. Clearly \eqref{eq:linear_in_the_halfplane_HD} may have non-zero 
solutions only for some specific choices of parameters (this is indeed the case according to Lemma \ref{lem:strong_coercivity}). For this reason, in this section we consider the number $k \geq 1$ and the parameter $\alpha \in \R$ as givens of the problem, and we look for pairs of numbers $(\mu,\omega) \in \R^2$ such that a solution $v$ as specified above exists.

The analysis that we have conducted in Section \ref{sec:sepvar} can be exploited to give a direct solution to this problem. Indeed we have the following result.

\begin{lemma} \label{lem:bessel_HD}
	For any $k \geq 1$, $\alpha \in \R$, there exists at least a value $\lambda\in \C$ satisfying
	\begin{equation}\label{eqn lambda ring}
		\begin{cases}
				\modB_{k-i\alpha}(\lambda)=0 \\
				\modB_{k-i\alpha}(t \lambda) \neq 0 &\forall t \in [0,1),
		\end{cases}
	\end{equation}
	where $\modB_\nu$ is defined in \eqref{eq:fintaBessel} for every $\nu\in\C$. For any such $\lambda $ the function
	\[
		v(x,y) =  e^{\alpha x -ky} \real\left(e^{i (k x+\alpha y)} \modB_{k-i\alpha}(\lambda  e^{-2y}) \right)
	\]
is a solution of \eqref{eq:linear_in_the_halfplane_HD}, with
	\[
		\omega = \frac{\imag(\lambda )}{k}, \qquad \mu = \alpha \frac{\imag(\lambda )}{k} - \real(\lambda ).
	\]
Moreover, there exists an analytic map $y\mapsto \zeta(y)$ such that
\[
v(x,y) = 0 
\qquad\iff\qquad
x=\zeta(y) + \frac{h\pi}{k},\qquad h\in\Z,
\]
 and
\[
	\zeta(y) = \frac{1}{k} (\beta - \alpha y) + o(1) \qquad \text{for some $\beta \in \R$ and $y \to +\infty$}.
\]
In particular, for any $y>0$, $v(\cdot, y)$ has exactly $2k$ zeros in each period $x\in[0,2\pi)$.
\end{lemma}
\begin{proof}
	The result is a direct consequence of Lemma \ref{lem:bessel}. We start by showing that for any choice of parameters, there exists at least a value $\lambda  \in \C$ verifying \eqref{eqn lambda ring}. Indeed, $\modB_{k-i\alpha}$ is a non constant analytic function with $\modB_{k-i\alpha}(0) \neq 0$, and it suffices to consider a zero $\lambda $ of $\modB_{k-i\alpha}$ with the least absolute value in order to guarantee that $\modB_{k-i\alpha}(t\lambda ) \neq 0$ for any $t\in [0,1)$. Of course, many (if not all) the zeros of $\modB_{k-i\alpha}$ may verify this assumption, but these constitute an at most countable discrete subset of $\C$.
	
	Exploiting the fact that the coefficients of \eqref{eq:linear_in_the_halfplane_HD} are real, we find that the function
	\begin{equation}\label{eq single}
		v(x,y) = e^{\alpha x} \real\left(  e^{i k x} D_k(y) \right)
	\end{equation}
	is a solution of \eqref{eq:linear_in_the_halfplane_HD}, where the function $D_k$ solves
	\begin{equation} \label{eq.Dy}
		\begin{cases}
			D_k''(y)  = \left[(k-i \alpha)^2 + \left( \omega \alpha - \mu +i \omega k \right) e^{-2y} \right] D_k(y), \qquad y > 0 \\
			D_k(0) = 0, \; D_k(y) \to 0 \text{ as $y \to +\infty$}.
		\end{cases}
	\end{equation}
	By Lemma \ref{lem:bessel}, equation \eqref{eq.Dy} is solved by any multiple of the function 
	\[
		y \mapsto e^{-(k-i\alpha)y}\modB_{k-i\alpha}\left((\omega \alpha - \mu +i \omega k)e^{-2y}\right),
	\]
	which in turns vanishes for $y \to +\infty$. The initial condition $D_k(0)=0$ is satisfied since we chose $\lambda  = \omega \alpha - \mu +i \omega k$ as a zero of the function $\modB_{k-i\alpha}$ (observe that we are negating \eqref{eq:cond_res}).
	
To conclude, we need to study the nodal properties of the function $v$. From its expression we readily see that for any fixed $y>0$, the function $x\mapsto v(x,y)$ has exactly $2k$ evenly spaced zeros in $[0,2\pi)$ since, by assumption, $\modB_{k-i\alpha}(\lambda  e^{-2y}) \neq 0$. From this we deduce also that the nodal lines of $v$ can be described, up to translations, by a function $y \mapsto \zeta(y)$. We notice that $\zeta$ is continuous by the implicit function theorem, as
	\[ 
		v(x, y) = 0 
		\iff
		\real\left(  e^{i k x} D_k(y) \right) = 0
	\]
and, for such $(x,y)$,
	\[
		\frac{\partial}{\partial x} \real\left(  e^{i k x} D_k(y) \right) = ik  \imag\left(  e^{i k x} D_k(y) \right) \neq 0.
	\]
More explicitly, 
writing 
\[
D_k(y)=\rho(y) e^{i\vartheta(y)}
\]
where $\rho(y)>0$ for $y>0$ and 
$\vartheta$ is an analytic lifting of the argument of $D_k$, we have that
	\[ 
		e^{\alpha x} v(x, y) = \real\left(  e^{i k x} D_k(y) \right) = 0
		\iff
	x - \frac{h\pi}{k}= \frac{1}{k}\left(\beta - 
	\vartheta(y)\right) =:\zeta(y).
	\]
Finally, the asymptotic behavior of $\zeta$ follows as in Lemma \ref{lem H1 and nodal}.
\end{proof}

We conclude with some additional remarks on the result. 
\begin{remark}
	(A question about uniqueness) If $v$ is a solution of \eqref{eq:linear_in_the_halfplane_HD}, then for any $A, \bar x \in \R$, the function $(x,y) \mapsto A v(x-\bar x, y)$ is again solution. We may wonder whether this family of functions completely describes the set of solutions of \eqref{eq:linear_in_the_halfplane_HD} under some additional condition (for instance that for any $x \in \R$, $v(x,y) \to 0$ as $y \to +\infty$). More precisely, fix $\omega$, $\mu$ and $\alpha$, in such a way that \eqref{eq:linear_in_the_halfplane_HD} admits at least a solution. Is this solution unique (up to translation in $x$ and multiplication by a real constant of course)? This seems to be a question of non trivial nature and it is related to the position of the zeros of Bessel functions with different order. From the proof of Lemma \ref{lem:bessel_HD} we can state the following: let $\alpha \in \R$ be such that for any $k_1, k_2 \geq 1$ and $z_1, z_2 \in \C$ we have
\[
\begin{cases}
I_{k_1-i\alpha}(z_1) = I_{k_2-i\alpha}(z_2) = 0 \\
\real(z_1^2) = \real(z_2^2) \\
\frac{\imag(z_1^2)}{k_1} = \frac{\imag(z^2_2)}{k_2}  
\end{cases}	\implies k_1 = k_2.
\]
Then for this specific value of $\alpha$ if \eqref{eq:linear_in_the_halfplane_HD} admits a solution, this solution is unique up to translation in $x$ and multiplication by a real constant.
\end{remark}

\begin{remark}
	(The symmetric case $\alpha = 0$) If $\nu \in \R$ and $\nu \geq 1$, the zeros of the modified Bessel function $I_\nu$ are purely imaginary numbers (and are given by $i j_{\nu,l}$, where $j_{\nu,l}$ is the $l$-th zero of the Bessel function $J_\nu$ with $l \in \N$). It follows that
	\[
		\modB_k(\lambda) = 0 \implies \lambda = - t^2 \quad \text{for some $t > 0$}.
	\]
	As a result, if $\alpha = 0$, then necessarily $\omega = 0$ (no rotation) and $\mu = j_{k,1}^2$. Since all the zeros belong to the same half-line spanning from the origin, the first non trivial zero is also the only one that verifies the assumptions of Lemma \ref{lem:bessel_HD}. We conclude that, in the case $\alpha = 0$, \eqref{eq:linear_in_the_halfplane_HD} has non-zero solutions only if $\mu = j_{k,1}^2$ and $\omega = 0$, and any solution (that converges to zero as $y\to+\infty$) is of the form
	\[
		v(x,y) = \left( A \cos(kx) + B \sin(kx) \right) J_k(j_{k,1}e^{-y})
	\]  
	for some $A,B \in \R$.
\end{remark}

\begin{remark}
	(The asymmetric case $\alpha \neq 0$) By Lemma \ref{lem:strong_coercivity}, and in particular \eqref{eq:_altre_coerc}, we already know that if $\alpha \neq 0$, for \eqref{eq:linear_in_the_halfplane_HD} to have a solution it is necessary that
	\[
	\mu \geq \left(j_{0 , 1}+\sqrt{k^2 + \alpha^2}\right)^2, 
	\]
From numerical explorations (see e.g.\ Figs.~\ref{fig:zeroes}, \ref{fig:3}), it seems that, if 
$\alpha \neq 0$, the zeros of the function $\modB_{k-i\alpha}$ belong to different lines spanning 
from the origin. In contrast with the case $\alpha = 0$, it thus seems to be the case that for 
$\alpha \neq 0$ \eqref{eq:linear_in_the_halfplane_HD} has infinitely many (but still countable 
many) solutions.
\end{remark}
\begin{figure}%[htbp]
\begin{center}
	\input{zeros_theta.pgf}
\caption{numerical zeroes of $\real\modB_{1-i}$ (blue) and $\imag\modB_{1-i}$ (red). The three zeroes located at $10.36+i23.66$, $20.22+i67.99$, $30.21+i132.04$, satisfy condition \eqref{eqn lambda ring}.}
\label{fig:zeroes}
\end{center}
\end{figure}
\begin{figure}[ht]
\begin{center}
	\input{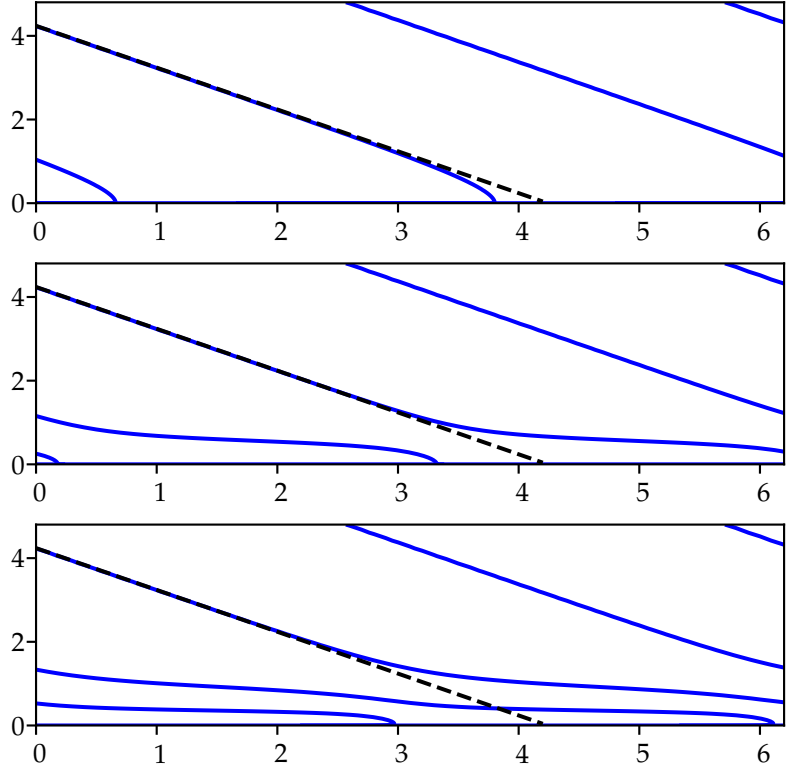}
\caption{nodal sets of the solutions corresponding to the three zeroes in Fig.~\ref{fig:zeroes}.}
\label{fig:3}
\end{center}
\end{figure}
%%
%% per aggiustare i pgfpicture, togliere ogni occorrenza di
%%
%% \color{textcolor}\sffamily\fontsize{10.000000}{12.000000}\selectfont
%%

\subsection{The homogeneous Neumann/Robin problem}	
Let $\sigma\in\R$. We consider the problem
\begin{equation}\label{eq:linear_in_the_halfplane_NH}
	\begin{cases}
		-\Delta v + \omega e^{-2y} v_x = e^{-2y}\mu v & x\in\R,\,y>0\\
		v(x + 2\pi, y) = e^{2\pi \alpha} v(x,y)              & x\in\R,\,y\ge 0\\
		\partial_y v(x,0) + \sigma v(x,0)=0 & 0\le x \le 2\pi,
	\end{cases}
\end{equation}
which entails Robin ($\sigma\neq0$) or Neumann ($\sigma=0$) boundary conditions.

As in the previous section we can find single-mode solutions that exhibit a precise nodal behavior.
\begin{lemma} \label{lem:bessel_HN}
	For any $k \geq 1$, $\alpha \in \R$, assume that there exists $\lambda \in \C$ satisfying
	\[
		\begin{cases}
			2\lambda \modB_{k-i\alpha}'(\lambda) + (k-i\alpha-\sigma)\modB_{k-i\alpha}(\lambda) = 0\\
			\modB_{k-i\alpha}(t \lambda) \neq 0 \qquad \forall t \in [0,1).
		\end{cases}
	\]
	Then we have that 
	\[
	v(x,y) = e^{\alpha x - ky} \real\left(e^{i (kx + \alpha y)}  \modB_{k-i\alpha}(\lambda e^{-2y}) \right)
	\]
	is a solution of \eqref{eq:linear_in_the_halfplane_NH} for the particular choice of parameters
	\[
		\omega = \frac{\imag(\lambda)}{k}, \qquad \mu = \alpha \frac{\imag(\lambda)}{k} - \real(\lambda).
	\] 
Moreover, the nodal set of $v$ has the same properties as those described in Lemma \ref{lem:bessel_HD}.
\end{lemma}
\begin{proof}
	We already know that any function of type
	\[
		v(x,y) = e^{\alpha x} \real\left( e^{ i k x} N_k(y) \right)
	\]
	is a solution of the differential equation in \eqref{eq:linear_in_the_halfplane_NH} provided that
	\[
		N_k''(y)  = \left[(k-i \alpha)^2 + \left( \omega \alpha - \mu +i \omega k \right) e^{-2y} \right] N_k(y), \qquad y > 0.
	\]	
	Once again we can appeal to Lemma \ref{lem:bessel} for an explicit expression for the function $N_k$. In order to impose the boundary condition at $y = 0$ we find
	\[
		N_k'(y) = \modB_{k-i\alpha}'(\lambda e^{-2y}) (-2\lambda e^{-2y}) e^{-(k-i\alpha) y} - \modB_{k-i\alpha}(\lambda e^{-2y}) (k-i\alpha)e^{-(k-i\alpha) y}
	\]
	that is
	\[
		N_k'(0) = \modB_{k-i\alpha}'(\lambda) (-2\lambda) - (k-i\alpha) \modB_{k-i\alpha}(\lambda) = 0.
	\]
	The rest of the proof follows easily.
\end{proof}

\subsection{Entire solutions}
Finally we consider the case of entire solutions, that is we look for functions $v$ that verify
\begin{equation}\label{eq:linear_in_the_plane}
	\begin{cases}
		-\Delta v + \omega e^{-2y} v_x = e^{-2y}\mu v \\
		v(x + 2\pi, y) = e^{2\pi \alpha} v(x,y)              
	\end{cases} \quad (x,y) \in \R^2,
\end{equation}
vanish for $y \to +\infty$ and, as before, change sign exactly $2k$ times ($k\geq 1$) in each period of length $2\pi$ in the $x$ direction. The similar considerations as before lead us to the following result.
\begin{lemma} \label{lem:bessel_E}
	Let $k \geq 1$, $\alpha \in \R$. Consider any $\lambda \in \C$ such that 
	\begin{equation}\label{eqn assump lambda entire}
	\modB_{k-i\alpha}(t \lambda) \neq 0 \qquad \forall t > 0.
	\end{equation}
	Then the function 
	\begin{equation}\label{eq entire}
	v(x,y) =e^{\alpha x-ky} \real\left(e^{i (k x+\alpha y)} \modB_{k-i\alpha}(\lambda e^{-2y}) \right)
	\end{equation}
	is a solution of \eqref{eq:linear_in_the_plane} for the particular choice of parameters 
	\[
		\omega = \frac{\imag(\lambda)}{k}, \qquad \mu = \alpha \frac{\imag(\lambda)}{k} - \real(\lambda).
	\]
\end{lemma}
Once again, we point out that $\modB_{k-i\alpha}$ is analytic and thus its has at most countably many zeros, meaning that, apart from a negligible set, any $\lambda \in \C$ gives rise to an entire solution.

In the case of entire solutions, it is interesting to study once again the shape of the nodal lines of the solutions, which now are defined also for $y<0$.

\begin{lemma}\label{lem:asy_entire}
	Let $v$ be the function \eqref{eq entire} in Lemma \ref{lem:bessel_E}, then there exists an analytic function $y \mapsto \zeta(y)$, defined for any $y \in \R$, such that
	\begin{itemize}
		\item $v(x,y) = 0$ if and only if $x = \zeta(y) + \frac{h\pi}{k}$, $y\in\R$, $h \in \Z$, 
			and consequently, in the regions $\{(x,y) : \frac{h\pi}{k} < x - \zeta(y) < \frac{(h+1)\pi}{k}\}$, for any $h \in \Z$, $v$ does not change sign;
	    \item for $y \to +\infty$, $\zeta$ is asymptotic to a line: there exists $\beta \in \R$ such that
		\[
			\zeta(y) = \frac{1}{k} (\beta - \alpha y)+o(1)\qquad\text{as }y\to+\infty;
		\]
	    \item for $y \to -\infty$, $\zeta$ is asymptotic to an exponential curve 
	    \[
			\zeta(y) = 
			\gamma e^{-y} + O(1)\qquad\text{as }y\to-\infty,
		\]
		where 
		\[
			\gamma = 
			\begin{cases}
			\frac{1}{k} \sign(\omega) \sqrt{\sqrt{\left(\frac{\omega\alpha-\mu}{2}\right)^2 
			+ \left(\frac{\omega k}{2}\right)^2} - \frac{\omega\alpha-\mu}{2}} & \omega\neq0\\
			0       &  \omega=0,\ \mu<0\\
			\frac{1}{k}\sign(\alpha)\sqrt{\mu}      &  \omega=0,\ \mu>0,\\
			\end{cases}
		\]
		unless $\omega = \mu = 0$, in which case 
		\[
		\zeta(y) = \frac{1}{k}\left(\beta - \alpha y\right)\qquad y\in\R,
	\]
	\end{itemize}
\end{lemma}
\begin{proof}
	The first conclusions of the result follow from similar (and much simpler) considerations as  in Proposition \ref{prop:nod} 
	and Lemma \ref{lem:bessel_HD}. We only study the asymptotic behavior of $\zeta$ 
	as $y\to-\infty$. As we shall see, beyond the validity of \eqref{eqn assump lambda entire}, we need to distinguish three cases, according to the different expansions of the Bessel functions at infinity:  (case 1) $\omega = \mu = 0$; (case 2) $\omega = 0$, $\mu > 0$; (case 3) either $\omega = 0$ and $\mu < 0$, or $\omega \neq 0$.
	
	\noindent \textbf{Case 1)} We start with the simplest case, that is  $\omega = \mu = 0$. This is equivalent to assuming that $\lambda = 0$, whence \eqref{eqn assump lambda entire} is automatically satisfied (recall that $\modB_{k-i\alpha}(0)\neq0$ for $k\ge1$). Substituting in \eqref{eq:linear_in_the_plane} we find that solutions are of the form
	\[
		v(x,y) = e^{\alpha x - k y} \cos(kx + \alpha y).
	\]
	In this case the nodal lines are described, up to translations, by the linear function
	\[
		\zeta(y) = \frac{1}{k}\left(\frac{\pi}{2} - \alpha y\right)\qquad y\in\R,
	\]
	and, in particular, the nodal set of $v$ is a family of parallel straight lines.
	
	\noindent \textbf{Case 2)} Next, we look at the case $\omega = 0$ and $\mu > 0$, which entails $\lambda = - \mu < 0$. We have that $\sqrt{\lambda} = -i\sqrt{\mu}$, where we have chosen the  determination of the square root with negative imaginary part. In this case, exploiting
	\eqref{eq entire}, \eqref{eq modB and I} and the relation between the Bessel functions and their modified versions, we have
	\[
		v(x,y) = e^{\alpha x} \left( \frac12 e^{ikx} J_\nu\left(\sqrt{\mu} e^{-y}\right) + \frac12 e^{-ikx} \overline{J_\nu\left(\sqrt{\mu} e^{-y}\right)}\right)
	\]
(to be precise, we take the line $y \mapsto \sqrt{\lambda}e^{-y}$ as path of monodromy for the 
determination of $J_\nu$).	In particular, from this expression we infer the necessary condition $\alpha \neq 0$: indeed, if $\nu = k \geq 1$, the Bessel function $J_k$ has all of its zeros on the  real  line, and thus we are contradicting \eqref{eqn assump lambda entire}. We have that (see \cite[p.~85]{MR0058756vol2})
	\[
		J_\nu(z) = \sqrt{\frac{2}{\pi z}} \left( \cos\left( z -\frac{\pi}{2}\nu -\frac{\pi}{4}\right) + O\left(\frac{1}{|z|}\right)\right) \qquad \text{for $|z| \to +\infty$ with $|\arg z| < \pi$}.
	\]
	As to what concerns us, we have that $z > 0$. Letting
	\[
		w = \sqrt{\mu} e^{-y} -\frac{\pi}{2}\nu -\frac{\pi}{4} = \left(\sqrt{\mu} e^{-y} -\frac{\pi}{2}k -\frac{\pi}{4}\right) + i \,\frac{\pi}{2} \alpha ,
	\]
	we may simplify the expression for $v$ and see that for $y \to -\infty$ the following asymptotic expansion holds
	\[		
	\sqrt{\frac{\pi \sqrt{\mu}}{2}} e^{-\alpha x - \frac12 y} v(x,y) = \frac12 e^{ikx} \cos w + \frac12 e^{-ikx} \cos \overline{w} + O(e^y).
	\]
	We point out that, in this peculiar case the solution $v$ decays for $y \to -\infty$ since $\imag(w)$ is bounded (constant). The last expression can be further simplified, since
	\[
	\begin{split}
		\frac12 e^{ikx} \cos w + \frac12 e^{-ikx} \cos \overline{w} & = \frac12 (\cos(kx) + i \sin(kx)) \cos w + \frac12  (\cos(kx) - i \sin(kx)) \cos\overline{w} \\
		&= \frac12 \cos(kx) \left[\cos w + \cos\overline{w} \right] +  \frac12 i \sin(kx) \left[\cos w - \cos\overline{w} \right]\\
		&= \cos(kx) \cos( \real w) \cosh (\imag w) + \sin(kx) \sin( \real w) \sinh (\imag w).
	\end{split}
	\]
	In order to determine the asymptotic behavior of the nodal lines of $v$, we need to solve the equation
	\[
		\cos(kx) \cos( \real w) \cosh (\imag w) + \sin(kx) \sin( \real w) \sinh (\imag w) = 0.
	\]
	It seems that this equation cannot be solved explicitly, nevertheless we can describe its set of solutions with sufficient accuracy for our purpose. In order to simplify the notation, we introduce the real function
	\begin{equation}\label{eqn F aux}
		F(X,Y) = \cos(X) \cos(Y) \cosh (T) + \sin(X) \sin(Y) \sinh (T)
	\end{equation}
	where we recall that the parameter $T = \imag w = \frac{\pi}{2}\alpha \neq 0$. In the plane $(X,Y) \in \R^2$, we want to describe the set $F(X,Y) = 0$. First of all, we point out that $F$ is $2\pi$-period both in $X$ and in $Y$ and enjoys the symmetries $F(X,Y) = F(Y,X)$, $F(-X,Y) = F(X,-Y)$, $F(X+\pi, Y) = F(X,Y+\pi) = -F(X,Y)$ and $F(-X,-Y) = F(X,Y)$ for any $(X,Y) \in \R^2$. In particular, we deduce that the equation $F(X,Y) = 0$ has infinitely many solutions and that for any fixed $Y \in \R$ (resp.~$X$) solutions of $F(X,Y)=0$ are equally spaced and of the form $X = X_Y+ h\pi$ for some given $X_Y \in \R$ and $h\in \Z$ (resp.~$Y = Y_X+ h\pi$). We deduce that for any given $Y \in [0,\pi)$ there exists a unique $X \in [0,\pi)$ such that $F(X,Y) =0$ and, vice versa.
	
	Next, let $(X_0, Y_0)\in \R^2$ such that $F(X_0, Y_0) = 0$, by the implicit function theorem the nodal set of $F$ is described locally at $(X_0,Y_0)$ by a function $X=Z(Y)$ if $\partial_X F(X_0, Y_0) \neq 0$. Arguing by contradiction, we have the system
	\[
		\begin{cases}
			\cos(X_0) \cos(Y_0) \cosh (T) + \sin(X_0) \sin(Y_0) \sinh (T) = 0\\
			\cos(X_0) \sin(Y_0) \sinh (T) - \sin(X_0) \cos(Y_0) \cosh (T) = 0
		\end{cases}
	\]
	which has a solution if and only if 
	\[
		\cos^2(Y_0) \cosh^2(T) + \sin^2(Y_0) \sinh^2(T) = 0.
	\]
	But this is impossible since $\cosh^2(T) \neq 0$ and $\sinh^2(T) \neq 0$ (recall that $T \neq 0$). Thus $\partial_X F(X_0, Y_0) \neq 0$ at any zero of $F$. Observe that we can perform similar computations exchanging variables and show that the function $Z$ is a bijection (and thus monotone). By periodicity, we can assume that $Z(0)=\frac{\pi}{2}$. We can determine the sense of monotonicity of $Z$ by computing $Z'(Y)$ for the zero $(X,Y) = (\frac{\pi}{2},0)$. We find
	\[
		Z'(0) = - \frac{\partial_Y F (\frac{\pi}{2},0)}{\partial_X F (\frac{\pi}{2},0)} = \tanh(T) =  \tanh\left(\frac{\pi}{2}\alpha\right)
	\]
	
	Bridging together the previous conclusions, we infer that
	\[
		0 \leq Z(Y) - \sign(\alpha) Y < \pi, \qquad \forall Y \in \R.
	\]
	Going back to the original variable, we find the asymptotic behavior
	\[
	\zeta(y) = \frac{1}{k}\sign(\alpha)\sqrt{\mu}e^{-y} + O(1)
	\qquad\text{as }y\to-\infty.
		\]	

	\noindent \textbf{Case 3)} We conclude with the third and last case, that is $\lambda = \omega \alpha - \mu + i \omega k \in \C \setminus \R_-$ together with \eqref{eqn assump lambda entire}. We recall that the modified Bessel function $I_\nu$ satisfies (see \cite[p.~86]{MR0058756vol2})
	\[
		I_\nu(z) = \frac{e^z}{\sqrt{2\pi z}} \left(1+ O\left(\frac{1}{|z|}\right)\right) \qquad \text{for $|z| \to +\infty$ with $|\arg z| < \frac{\pi}{2} - \delta$}.
	\]
	By \eqref{eq modB and I}, the entire function in \eqref{eq entire} is equal to 
	\[
		v(x,y) = e^{\alpha x } \real\left( e^{ikx} I_\nu\left(\sqrt{\lambda} e^{-y}\right)\right)
	\]
	where we choose as determination of the square root of $\lambda$ the one with strictly positive real part (recall that $\lambda \in \C \setminus \R_-$). Then $|\arg \sqrt\lambda | < 
	\frac{\pi}{2} - \delta$, for some $\delta>0$. We find
	\[
	\begin{split}
		v(x,y) &= e^{\alpha x} \real \left( e^{ikx} \frac{e^{\sqrt{\lambda} e^{-y}}}{\sqrt{2\pi \sqrt{\lambda} e^{-y}}} \left(1 + O(e^y)\right)\right) = e^{\alpha x } \real \left(  C_\lambda e^{ikx + \frac12 y + \sqrt{\lambda} e^{-y}} \left(1 + O(e^y)\right) \right) \\
		\\
		&=  e^{\alpha x +\frac12 y + \real(\sqrt{\lambda}) e^{-y}} \real \left( C_\lambda e^{ikx + i \imag\sqrt{\lambda} e^{-y} + iO(e^y)} \left|1 + O(e^y)\right|\right)= 0,
	\end{split}
	\]
	which in turns gives the asymptotic equation as $y \to -\infty$
	\[
		kx + \imag(\sqrt{\lambda}) e^{-y} + O(e^y) = \beta
 	\]
 	where $\beta \in \R$ and
 	\[
 		 \imag(\sqrt{\lambda}) = \sign(\omega k) \sqrt{\sqrt{\left(\frac{\omega\alpha-\mu}{2}\right)^2 + \left(\frac{\omega k}{2}\right)^2} - \frac{\omega\alpha-\mu}{2}}
 	\]
 	(with $\imag(\sqrt{\lambda}) = 0$ in case $\omega = 0$).
 	Notice that the sign above agrees with the fact that the nodal lines of the solution $v$ are spanned by monotone functions (see the proof of Lemma \ref{lem:bessel_HD}).
\end{proof}

\begin{remark}\label{rem:final}
In view of the results of Section \ref{sec:disk2halfplane} we have that any solution constructed 
in this section corresponds to an element of the corresponding class $\Scal_\rot$. In 
particular, if $\alpha=0$, we obtain (positive and negative parts of) smooth rotating solutions of 
the heat equation, with or without reaction term. Moreover, Lemma \ref{lem:asy_entire} 
provides a description of their nodal lines, which behave 
like arithmetic spirals of equation $\vartheta=\gamma r$ as $r\to+\infty$, as we claimed in Remark \ref{rem_entire_intro}.
\end{remark}

\appendix

\section{Weighted embeddings and Poincar\'{e} inequalities}

In this appendix we give the proof of some results cited in the paper for the sake of completeness. We start with a very classical compact embedding result.

\begin{lemma}\label{lem comp emb}
The functional space $H^1_0(\R^+;\C)$ embeds compactly in
\[
	L = \left\{U \in L^1_\loc(\R^+;\C) : \|U\|_L^2 = \int_{y>0} e^{-2y} |U|^2 < +\infty \right\}.
\]
\end{lemma}
\begin{proof}
Let $\{u_n\}_{n\in\N} \subset H^1_0(\R^+;\C)$ be a weakly converging sequence and let $u$ be its limit. Since the embedding of $H^1_0$ in $L$ is clearly continuous, $u_n \rightharpoonup u$ in $L$ and in order to show that $u_n \to u$ in $L$ we just need to prove the convergence of the norms. Let
\[
    d_n = \left|\int_{y>0} e^{-2y} u_n^2 - \int_{y>0} e^{-2y} u^2 \right|.
\]
Observe that $\{d_n\}_n$ is a positive sequence. We have that
\[
\begin{split}
    d_n &\leq \int_{y>0} e^{-2y} |u_n^2- u^2| = \int_{0}^T e^{-2y} |u_n^2- u^2| + \int_{T}^\infty e^{-2y} |u_n^2- u^2|\\
    &\leq \int_{0}^T e^{-2y} |u_n^2- u^2| + e^{-2T} ( \|u_n\|_{L^2}^2 + \| u\|_{L^2}^2) \leq \int_{0}^T e^{-2y} |u_n^2- u^2| + 2C e^{-2T}
\end{split}
\]
for any $T>0$. Since $H^1(0,T)$ is compactly embedded in $L^2(0,T)$, we conclude that there exists $\{\eps_{n,T}\}_n$ such that $\eps_{n,T} \to 0$ and
\[
    d_n \leq \eps_{n,T} + 2 C e^{-2T}.
\]
To conclude, for any given $\delta > 0$ we can find $T > 0$ such that $C e^{-2T} < \delta/2$ and subsequently $\bar n$ such that $\eps_{n,T} \leq \delta / 2$ for any $n \geq \bar n$. This implies that for any $n \geq \bar n$ we have that $0 \leq d_n \leq \delta$, that is
\[
    \lim_{n \to +\infty} d_n = 0 \implies \int_{y>0} e^{-2y} u^2 = \lim_{n\to+\infty} \int_{y>0} e^{-2y} u_n^2
\]
and thus we conclude the strong convergence of the sequence $\{u_n\}_{n\in\N}$.
\end{proof}

Exploiting this compact embedding we can show the following weighted Poincar\'{e} inequality.
\begin{lemma}\label{lem coer}
    Let $a >0$ and $b \in \R$, then
    \[
        \int_{y> 0} |u'|^2 + (a^2-b e^{-2y})u^2 \geq 0
    \]
    for any $u \in H^1_0(\R^+)$ as long as
    \[
        b \leq (j_{a,1})^2
    \]
    where $j_{a,1}$ is the first (positive) zero of the Bessel function of first kind of order $a$.
\end{lemma}
\begin{proof}
The statement is equivalent to proving that
\begin{equation}\label{eqn minim}
    (j_{a,1})^2 = \inf_{u \in H^1_0(\R^+)} \left\{ \int_{y> 0} |u'|^2 + a^2 u^2 : \int_{y>0} e^{-2y} u^2 = 1 \right\}.
\end{equation}
The existence of a minimizers $u \in H^1_0(\R^+)$ follows directly from the embedding in Lemma \ref{lem comp emb}. As the functional and the constraint are even, we can assume that the minimizer $u$ is positive. Standard regularity results imply that the function $u$ is also smooth and strictly positive in $\R^+$. Let $\lambda \geq 0$ be the minimum of \eqref{eqn minim}. We have that $u \in H^1_0(\R^+)$ is a solution of
\[%begin{equation}
    \begin{cases}
        - u'' +(a^2 - \lambda e^{-2y}) u = 0\\
        u(0) = 0, \; u(y) > 0 \text{ for $y > 0$}.
    \end{cases}
\]%end{equation}
We argue as in Lemma \ref{lem:bessel}. We look for a solution defined by the series
\[
    u(y) = \sum_{n \geq 0} c_n e^{-(2n+a) y} \qquad \text{ where $c_n \in \R$ for $n \in \N$}
\]
We first make some formal computations, plugging this expression directly into the equation. We find that the coefficients $c_n$ must satisfy the following recursive relation for $n \geq 1$
\[
     c_{n} (2n+a)^2 = c_n a^2 - c_{n-1}  \lambda
\]
which is verified for instance by letting
\[
    c_{n} = \frac{(-1)^n}{n!\Gamma(n+1+a)} \left(\frac{\sqrt{ \lambda }}{2} \right)^{2n+a} \qquad \forall n \in \N
\]
thus leading us to the solution
\[
    u(y) = \sum_{n\in\N}\frac{(-1)^n}{n!\Gamma(n+1+a)} \left(\frac{\sqrt{ \lambda }}{2} e^{-y} \right)^{2n+a} = J_{a}\left(\sqrt{ \lambda } e^{-y}\right).
\]
We recall that if $a > 0$, then $J_a(0) = 0$. This gives that for any $a > 0$
\[
    \lim_{y\to+\infty}u(y) = 0.
\]
One can easily check that the series does converge in $H^1(\R^+)$ to its sum $u$. We only need to ensure that
\[
    u(0) = 0 \qquad \text{and} \qquad u(y) > 0 \quad \text{for any $y > 0$}.
\]
In terms of the function $J_a$ these conditions together mean that $\sqrt{ \lambda }$ has to be the first (positive) zero for $J_a$, that is
\[
    \sqrt{ \lambda } = j_{a,1} \iff  \lambda = (j_{a,1})^2. \qedhere
\]

\end{proof}

We can also show a similar Poincar\'{e} inequality for semi-infinite rectangles.
\begin{lemma}\label{lem eigen rect}
For any $a > 0$ and $b \in \R$ we consider the semi-infinite rectangle
\[
	Q_{a,b} = \left(-a/2, a/2 \right) \times (b, +\infty)
\]
and the corresponding functional space
\[
    H^1_0(Q_{a,b}) = \left\{ u \in H^1(Q_{a,b}) : u = 0 \text{ on } \partial Q_{a,b} \right\}.
\]
We have
\[
   \inf_{u \in H^1_0(Q_{a,b})} \left\{ \int_{Q_{a,b}} |\nabla u|^2 : \int_{Q_{a,b}} e^{-2y} u^2 = 1 \right\} = e^{2b} \left(j_{\frac{\pi}{a}, 1}\right)^2.
\]
\end{lemma}
\begin{proof}
By the same compactness argument of Lemma \ref{lem comp emb}, we can show that the infimum is attained by a function $u \in H^1_0 (Q_{a,b})$ which,  by standard results, is also positive and smooth in $Q_{a,b}$. Up to a translation in $y$, the function $u$ is then a positive solution of
\[
    \begin{cases}
        -\Delta u = \lambda e^{-2b} e^{-2y} u &\text{in $Q_{a,0}$}\\
        u = 0 &\text{on $\partial Q_{a,0}$}
    \end{cases}
\]
for some $\lambda \geq 0$. By separation of variable we can easily show that $u$ is of the form
\[
    u(x,y) = \cos\left(\frac{\pi}{a} x \right) v(y)
\]
where the new unknown function $v \in H^1_0(\R^+)$ solves
\[
    \begin{cases}
        - v'' + \left(\frac{\pi^2}{a^2} - \lambda e^{-2b} e^{-2y}\right) v = 0\\
        v(0) = 0, \; v(y) > 0 \text{ for $y > 0$}.
    \end{cases}
\]
By Lemma \ref{lem coer} we conclude that
\[
    \lambda  e^{-2b} =  \left(j_{\frac{\pi}{a},1}\right)^2. \qedhere
\]
\end{proof}

\section*{Acknowledgments} 
G.V.~acknowledges support from the project Vain-Hopes within
the program VALERE - Universit\`a degli Studi della Campania ``Luigi Vanvitelli'', by the Portuguese
government through FCT/Portugal under the project PTDC/MAT-PUR/1788/2020

A.Z.~acknowledges support from the ANR via the project 
%\href{https://www.math.univ-toulouse.fr/~gfaye/anrIndyana.html}
{Indyana} under grant agreement ANR-21-CE40-0008 and the project SHAPO under grant 
agreement ANR-18-CE40-0013. 

Work partially supported by the INdAM - GNAMPA group.

%\bibliography{segregation}
%\bibliographystyle{abbrv}

\end{document}